\pdfoutput=1
\documentclass[11pt,reqno,english]{amsart}
\usepackage{amssymb,amsmath,amsthm,verbatim}
\usepackage[margin=1in, footskip=0.3in]{geometry}
\usepackage[dvipsnames]{xcolor}
\usepackage{url}
\usepackage{mathrsfs}
\usepackage[all]{xy}
  \SelectTips{cm}{10}
  \everyxy={<2.5em,0em>:}
\usepackage{tikz}
\usepackage{picinpar} 

\usepackage{fancyhdr}
\usepackage{ifpdf}
  \ifpdf
    \PassOptionsToPackage{hyphens}{url}\usepackage[colorlinks]{hyperref}
 \hypersetup{
  colorlinks,
  citecolor=Mulberry,
  linkcolor=Red,
  urlcolor=Green}
  \else
    
    \newcommand{\href}[2]{#2}
  \fi

\theoremstyle{plain}
  \newtheorem{lemma}[equation]{Lemma}
  \newtheorem{proposition}[equation]{Proposition}
  \newtheorem{theorem}[equation]{Theorem}
  
    \newtheorem{question}[equation]{Question}

\theoremstyle{definition}
  \newtheorem{definition}[equation]{Definition}

\renewcommand{\thesection}{\arabic{section}}
\renewcommand{\theequation}{\thesection.\arabic{equation}}

 \DeclareFontFamily{U}{manual}{}
 \DeclareFontShape{U}{manual}{m}{n}{ <->  manfnt }{}
 \newcommand{\manfntsymbol}[1]{%
    {\fontencoding{U}\fontfamily{manual}\selectfont\symbol{#1}}}

\makeatletter
   \@addtoreset{section}{part}
   \@addtoreset{equation}{section}
   \@addtoreset{footnote}{section}

    {\hspace*{\fill}$\lrcorner$\endgraf\endgroup\end{trivlist}}
 \newenvironment{example}[1][]{
   \refstepcounter{equation}
   \begin{proof}[Example~\theequation%
   \@ifnotempty{#1}{ (#1)}.]
   }
  {\end{proof}}
 \newenvironment{remark}[1][]{
   \refstepcounter{equation}
   \begin{proof}[Remark~\theequation%
   \@ifnotempty{#1}{ (#1)}.]
   }
  {\end{proof}}
\makeatother

  \DeclareFontFamily{OT1}{pzc}{}
  \DeclareFontShape{OT1}{pzc}{m}{it}{<-> s * [1.100] pzcmi7t}{}
  \DeclareMathAlphabet{\mathpzc}{OT1}{pzc}{m}{it}

\newif\ifhascomments \hascommentstrue
\ifhascomments
  \newcommand{\wei}[1]{{\color{purple}[[\ensuremath{\bigstar\bigstar\bigstar} #1]]}}
  \newcommand{\matt}[1]{{\color{red}[[\ensuremath{\spadesuit\spadesuit\spadesuit} #1]]}}
  \newcommand{\either}[1]{{\color{blue}[[\ensuremath{\spadesuit\spadesuit\spadesuit} #1]]}}
\else
  \newcommand{\wei}[1]{}
  \newcommand{\matt}[1]{}
  \newcommand{\either}[1]{}
\fi

\newcommand{\<}{\langle}
\renewcommand{\>}{\rangle} 

\newcommand{\A}{\mathcal A}

\DeclareMathOperator{\aut}{Aut}
\DeclareMathOperator{\Aut}{\ensuremath{\mathcal{A}\kern-.125em\mathpzc{ut}}}

\newcommand{\bbar}[1]{\setbox0=\hbox{$#1$}\dimen0=.2\ht0 \kern\dimen0 \overline{\kern-\dimen0 #1}}

\makeatletter
\newcommand{\extp}{\@ifnextchar^\@extp{\@extp^{\,}}}
\def\@extp^#1{\mathop{\bigwedge\nolimits^{\!#1}}}
\makeatother

\DeclareMathOperator{\End}{End}
\DeclareMathOperator{\Endo}{\ensuremath{\mathcal{E}\kern-.125em\mathpzc{nd}}}

\newcommand{\Fsep}{F^{\sep}}

\DeclareMathOperator{\GL}{GL}
\newcommand{\Gal}{\mathrm{Gal}}
\newcommand{\HH}{\mathcal H}

\DeclareMathOperator{\Hom}{\ensuremath{\mathcal{H}\kern-.125em\mathpzc{om}}}
\newcommand{\id}{\mathrm{id}}
\DeclareMathOperator{\Ind}{Ind}

\DeclareMathOperator{\Mat}{Mat}

\DeclareMathOperator{\No}{N}

\newcommand{\QQ}{\mathbb Q}
\newcommand{\RR}{\mathbb R}
\DeclareMathOperator{\Res}{Res}

\DeclareMathOperator{\sep}{sep}
\DeclareMathOperator{\SL}{SL}

\DeclareMathOperator{\Spur}{Spr}

\DeclareMathOperator{\Sym}{Sym}

\DeclareMathOperator{\sgn}{sgn}

\DeclareMathOperator{\Tr}{Tr}

\newcommand{\ZZ}{\mathbb{Z}}

 \def\ari[#1]{\ar@{^(->}[#1]}
 \def\are[#1]{\ar[#1]^{\txt{\'et}}}
 \def\areh[#1]{\ar[#1]|{\txt{$H$-eq}}^{\txt{\'et}}}
 \def\ars[#1]{\ar@{->>}[#1]}
 \newcommand{\dplus}{\ar@{}[d]|{\mbox{$\oplus$}}}
 \newcommand{\dtimes}{\ar@{}[d]|{\mbox{$\times$}}}

\newtheorem*{introtheorem}{Theorem}

\begin{document}

\title{Galois closures of non-commutative rings and an application to Hermitian representations}

\author{Wei Ho}
\thanks{The first author was supported by NSF grant DMS-1701437.}
\address{Department of Mathematics, University of Michigan, Ann Arbor, MI 48109.}
\email{weiho@umich.edu}

\author{Matthew Satriano}
\thanks{The second author was partially supported by 
an NSERC Discovery grant.}
\address{Pure Mathematics, University of Waterloo, 200 University Avenue West, Waterloo, Ontario, Canada N2L 3G1}
\email{msatrian@uwaterloo.ca}

\date{\today}

\begin{abstract}
Galois closures of commutative rank $n$ ring extensions were introduced by Bhargava and the second author. In this paper, we generalize the construction to the case of non-commutative rings. We show that non-commutative Galois closures commute with base change and satisfy a product formula. As an application, we give a uniform construction of many of the representations arising in arithmetic invariant theory, including many Vinberg representations.
\end{abstract}

\maketitle
\thispagestyle{fancy}
	
\tableofcontents

\section{Introduction}
In the last fifteen years, there have been many beautiful applications given by interpreting orbit spaces of representations as moduli spaces of arithmetic or algebraic objects, such as ideal classes of low rank rings or Selmer elements of elliptic curves. Many of the representations that arise seem to be closely related to one another, and in some cases, they can be formally related to one another by a process called Hermitianization; see \cite{hcl1,coregular,pollack}. In this paper, we construct such representations via a uniform approach. Our method relies on a seemingly unrelated problem: defining Galois closures of possibly non-commutative rings.

Galois closures of commutative rank $n$ ring extensions were studied in \cite{gc}, building on previous work of Grothendieck \cite[Expos\'e 4]{chevalley}, Katz--Mazur \cite[\S1.8.2]{katzmazur}, and Gabber \cite[\S5.2]{ferrand}. Given a morphism $R\to A$ of commutative rings realizing $A$ as a free $R$-module of rank $n$, the \emph{Galois closure} $G(A/R)$ is defined as the quotient $A^{\otimes n}/I_{A/R}$, where $I_{A/R}$ is an ideal generated by relations coming from characteristic polynomials. More precisely, given $a\in A$, consider the $R$-linear endomorphism of $A$ given by multiplication by $a$ and let $T^n + \sum_{j=1}^n (-1)^j s_{A,j}(x) T^{n-j}$ be its characteristic polynomial. Let $a^{(i)}\in A^{\otimes n}$ denote $1\otimes\cdots\otimes a\otimes\cdots1$ where $a$ is in the $i$-th tensor factor, and let $e_j$ denote the $j$-th elementary symmetric function. Then the ideal $I_{A/R}$ is generated by the relations
\[
e_j(a^{(1)},a^{(2)},\ldots,a^{(n)}) - s_{A,j}(a),
\]
as $a$ runs through all elements of $A$. Then $G(A/R)$ is an $R$-algebra equipped with a natural $S_n$-action, and the elements $a^{(1)},a^{(2)},\ldots,a^{(n)}$ behave as if they are ``Galois conjugates.'' One key property is that $G(A/R)$ commutes with base change on $R$. This construction has since been generalized by Gioia to so-called intermediate Galois closure \cite{gioia} as well as by Biesel to Galois closures associated to subgroups of $S_n$ \cite{biesel-thesis}.

We now describe the connection between Galois closures of non-commutative algebras and problems in arithmetic invariant theory. In this paper, we obtain many of the representations with arithmetic applications by the following uniform construction: let $A$ be a possibly non-commutative degree $n$ $R$-algebra and let $G(A/R)$ be its Galois closure, as we define in Section \ref{sec:Galclosure}, which comes with a natural $S_n$-action. 
For an $n$-dimensional  $m\times m\times\cdots\times m$ array with entries in $G(A/R)$, there are two natural $S_n$-actions: one on $G(A/R)$, and the other permuting the coordinates of the $n$-dimensional array. The subspace where these two actions coincide has a natural action of the matrix ring $\Mat_m(A) \otimes G(A/R)$; we refer to this as the associated {\em Hermitian representation} $\HH_{A,m}$. See Section \ref{sec:Hermitian} for lists of representations obtained in this manner that have arisen in arithmetic invariant theory. We hope that our uniform construction of these representations $\HH_{A,m}$, with just the input of a degree $n$ $R$-algebra $A$ and a positive integer $m$, will also give a systematic approach to studying the moduli problems related to the orbit spaces of the Hermitian representations.

For $n$ and $m$ sufficiently small, these Hermitian representations were studied explicitly in previous work \cite{hcl1,coregular,pollack}. Galois closures were not needed in these previous papers for two main reasons. First, when $m$ is small, the entries of the elements in the Hermitian representation may be defined over $A$ itself. And second, when $n$ is small, the Galois closure $G(A/R)$ is quite simple; for example, when $A$ is a quadratic algebra ($n = 2$), the Galois closure $G(A/R)$ is isomorphic to $A$ (see Proposition \ref{prop:quadratic}). Similarly, when $A$ is a decomposable cubic algebra, e.g., $A=R\times B$ for $B$ quadratic, then $G(A/R)\simeq B^{\oplus3}$ (see Proposition \ref{prop:k-B}), so $m \times m \times m$ arrays Hermitian with respect to $A = R \times B$ may be viewed as an $m$-tuple of $m \times m$ matrices Hermitian with respect to $B$ (see Example \ref{ex:mxmxm as m tuple of mxm}). Thus, the need for Galois closures in describing Hermitian representations does not arise until one considers $m \geq 3$ and indecomposable cubic algebras $A$, such as the matrix ring $\Mat_3(R)$.

The aforementioned Proposition \ref{prop:k-B} and Example \ref{ex:mxmxm as m tuple of mxm} are both specific cases of more general results as we now discuss. In Section \ref{sec:properties}, we prove the following product formula, which allows one to calculate the Galois closure of decomposable algebras in terms of the Galois closures of its components.

\begin{introtheorem}[Product formula] \label{thm:product-formula-intro}
For $1 \leq i \leq k$, let $A_i$ be a degree $n_i$ $R$-algebra. Then
\begin{equation*}
G(A_1 \times \cdots \times A_k/R) \simeq \left( G(A_1/R)\otimes\cdots\otimes G(A_k/R) \right)^N
\end{equation*}
where $N$ is the multinomial coefficient ${n\choose n_1,\ldots,n_k}$.
\end{introtheorem}

As a consequence, in Theorem \ref{thm:prod-form-Herm}, we may write the Hermitian representation of a decomposable algebra in terms of the Hermitian representations of its components.

\begin{introtheorem}[Product formula for Hermitianizations] \label{thm:product-formula-herm-intro}
For $1\leq i\leq k$, let $A_i$ be a degree $n_i$ $R$-algebra, and let $A = \prod_{i=1}^k A_i$. Then
for any positive integer $m$, we have
$\HH_{A,m}\simeq\HH_{A_1,m}\otimes\cdots\otimes\HH_{A_k,m}$.
\end{introtheorem}

We also prove that taking Galois closures commutes with base change.

\begin{introtheorem}[Base change] \label{thm:base-change-intro}
Let $A$ be a degree $n$ $R$-algebra and let $S$ be a commutative $R$-algebra. The base change map yields an isomorphism
\[
G(A/R)\otimes_RS\simeq G((A\otimes_R S)/S).
\]
\end{introtheorem}

In this paper, all the algebras we consider are associative, but we believe it would be useful to generalize these ideas to non-associative algebras such as cubic Jordan algebras as well, especially as Jordan algebras have already been crucially used in basic examples of Hermitianization.

\vspace{0.7em}
\noindent{\bf Acknowledgments.} It is our pleasure to thank Jason Bell, Manjul Bhargava, Bhargav Bhatt, Owen Biesel, Jonah Blasiak, Andrei Caldararu, Aaron Pollack, and Paul Smith for helpful comments and discussions.

\section{Galois closures of non-commutative rings} \label{sec:Galclosure}
In this section, we define the {\em Galois closure} for certain classes of (possibly noncommutative) rings and discuss several properties. When $A$ is commutative, it recovers the construction in \cite{gc}.

\subsection{Degree $n$ algebras}
We first define a degree $n$ algebra over a commutative ring $R$.
We use the following notion: a morphism $f\colon M\to N$ of $R$-modules is \emph{universally injective} if for every commutative $R$-algebra $S$, the induced map $f\otimes1\colon M\otimes_R S\to N\otimes_R S$ is injective. For example, if $f$ is split, then $f$ is universally injective.

\begin{definition}
\label{def:degree-n-algebras}
Let $A$ be a central $R$-algebra that is free of finite rank as an $R$-module. Let $R'$ be a finitely generated commutative $R$-algebra such that $R \to R'$ is universally injective, with a universally injective $R$-algebra homomorphism $\iota\colon A\to\Mat_n(R')$. We say that the triple $(A,R',\iota)$ is a {\em degree $n$ $R$-algebra} if for all $a\in A$, the characteristic polynomial 
\begin{equation} \label{eq:monicpoly}
P_{A,a}(T) = \det(T-\iota(a)) = T^n - s_{A,1}(a) T^{n-1} + \cdots + (-1)^n s_{A,n}(a)
\end{equation}
lives in $R[T]$, in other words, $s_{A,i}(a) \in R$ for $1 \leq i \leq n$. We frequently suppress $R'$ and $\iota$ from the notation if they are unambiguous, and refer to $A$ itself as a degree $n$ $R$-algebra.
\end{definition}

For a degree $n$ $R$-algebra $(A,R', \iota)$, we refer to $\Tr(a) := s_{A,1}(a)$ and $\No(a) := s_{A,n}(a)$ respectively as the {\em trace} and {\em norm} of $a \in A$. It is immediate from the definition that $s_{A,j}(ra)=r^js_{A,j}(a)$ for all $r\in R$ and $a\in A$; in particular, the trace is additive and the norm is multiplicative.

\begin{remark}
If $R' = R$, clearly the required property $P_{A,a}(T) \in R[T]$ is automatically satisfied.
\end{remark}

\begin{remark}
Let $(A,R', \iota)$ be a degree $n$ $R$-algebra and let $\psi\colon\Mat_n(R')\to\Mat_n(R')$ be conjugation by an element of $\GL_n(R')$. Then $(A,R',\psi\iota)$ is a degree $n$ $R$-algebra and all the characteristic polynomials $P_{A,a}(T)$ for the two algebras coincide.
\end{remark}

\begin{example}[Left multiplication]
\label{ex:char-polyn-generic-polyn}
Let $A$ be a central $R$-algebra that is free of finite rank $n$ as an $R$-module. Then we can view $A$ as a degree $n$ $R$-algebra as follows. Choose a basis $u_1, \ldots, u_n$ for $A$ over $R$. Left multiplication by elements $a\in A$ induces a natural map $\iota \colon A \to \Mat_n(R)$. Then $\iota$ is clearly injective since $a \cdot 1 \neq 0$ if $a \neq 0$. It is in fact universally injective by the same observation, since $u_1, \ldots, u_n$ is also a basis for $A \otimes_R S$ over $S$, for any commutative $R$-algebra $S$.
\end{example}

\begin{example}[Split algebra of any degree]
\label{ex:split-alg}
For any $n \geq 1$, we may view $R^n$ as a degree $n$ $R$-algebra by left multiplication as in Example \ref{ex:char-polyn-generic-polyn}. In this case, the element $a = (r_1, \ldots, r_n)$ has characteristic polynomial $P_{R^n,a}(T) = \prod_{i=1}^n (T-r_i)$.
\end{example}

\begin{example}[Trivial algebra of any degree]
\label{ex:triv-generic-polyn}
For any $n\geq1$, we may view $R$ as a degree $n$ algebra over itself by choosing $\iota \colon R \to \Mat_n(R)$ as the diagonal embedding of $R$. Then $P_{R,a}(T) = (T-a)^n$ for any $a \in R$. Although seemingly trivial, this example plays a useful role.

More generally, for any degree $n$ algebra $(A,R',\iota)$ and integer $m \geq 1$, we may also give $A$ the structure of a degree $mn$ algebra via the diagonal block map $(\iota,\ldots,\iota) \colon A \to \Mat_{mn}(R')$. The characteristic polynomials are $m$-th powers of the original characteristic polynomials.
\end{example}

\begin{example}[Matrix algebras]
\label{ex:matrix-alg-polyn}
If $A = \Mat_n(R)$, then taking $\iota$ to be the identity map gives $A$ the structure of a degree $n$ algebra.
\end{example}

\begin{example}[Central simple algebras]
\label{ex:central-simple-alg-polyn}
If $A$ is a central simple algebra over a field $F$, then there exists a splitting field $K$ over $F$ such that $A \otimes_F K \simeq \Mat_n(K)$, where $n$ is the square root of the rank of $A$ as a $F$-vector space. We thus have an injection $\iota \colon A \to \Mat_n(K)$, and it is universally injective because $\iota$ is split. The polynomials $P_{A,a}$ agree with the {\em reduced characteristic polynomial} of a central simple algebra $A$, and it is well known that the coefficients lie in $F$ (see, e.g., \cite[\S IV.2]{berhuyoggier}). In this way, we may view $A$ as an algebra of degree equal to the square root of the rank of $A$ (which is the typical definition of the degree of a central simple algebra). 
\end{example}

\begin{remark}
One may generalize the definition of degree $n$ algebra to {\em locally} free $R$-modules of finite rank by requiring a universally injective homomorphism to an endomorphism algebra of a rank $n$ vector bundle over $R'$ instead. The definitions of and theorems for Galois closures will also generalize in a similar way, but we focus on the case of free $R$-modules in the rest of the paper for simplicity.
\end{remark}

We next introduce products of degree $n$ $R$-algebras.

\begin{definition}
\label{def:products of degree n Ralgs}
Let $(A_1,R_1,\iota_1)$ and $(A_2,R_2,\iota_2)$ be $R$-algebras of degrees $n_1$ and $n_2$, respectively. We define the product $(A_1,R_1,\iota_1)\times(A_2,R_2,\iota_2)$ to be the degree $n_1+n_2$ $R$-algebra $A=A_1\times A_2$ with the universally injective
composition
$$\iota \colon A_1 \times A_2 \stackrel{(\iota_1,\iota_2)}{\longrightarrow} \Mat_{n_1}(R_1) \times \Mat_{n_2}(R_2) \hookrightarrow \Mat_{n_1}(R_1 \times R_2) \times \Mat_{n_2}(R_1 \times R_2) \hookrightarrow \Mat_{n_1 + n_2} (R_1 \times R_2)$$
where the last injection is given by block diagonals.
If $a_i \in A_i$ has characteristic polynomials $P_{A_i,a_i}$ for $i\in\{1,2\}$, then the characteristic polynomial of $a = (a_1, a_2) \in A_1 \times A_2$ is clearly the product $P_{A,a}(T)=P_{A_1,a_1}(T)P_{A_2,a_2}(T) \in R[T]$.
\end{definition}

The following gives some further properties of product $R$-algebras.

\begin{lemma}
\label{l:sm for products}
Let $A_i$ be a degree $n_i$ $R$-algebra for $1 \leq i \leq k$, and endow $A = \prod_{i=1}^k A_i$ with its associated structure as a degree $n:=\sum_{i=1}^k n_i$ $R$-algebra.
\begin{enumerate}
\item \label{sm prods::sm}
If $a=(a_1,\ldots,a_k)\in A$, then
\[
s_{A,m}(a)=\sum_{\substack{0 \leq m_i \leq n_i \\ m_1+\cdots+m_k=m}}\prod_{i=1}^k s_{A_i,m_i}(a_i),
\]
where we set $s_{A_i,0}(a_i) = 1$.
\item \label{sm prods::sm special case} 
If $a = (0,\ldots, 0, a_j, 0, \ldots 0)$, then
\[
s_{A,m}(a)=s_{A_j, m}(a_j).
\]
\end{enumerate}
\end{lemma}
\begin{proof}
We first show (\ref{sm prods::sm}). We use the notation $[T^j]Q$ to denote the $T^j$-coefficient of a polynomial $Q(T)$. As $P_{A,a}(T)=\prod_j P_{A_j,a_j}(T)$, we have
\[
[T^{n-m}]P_{A,a}=\sum_{\substack{0 \leq m_i \leq n_i-m_i \\ m_1+\cdots+m_k=n-m}} \prod_j [T^{m_i}]P_{A_i,a_i}.
\]
Since the indices of the summation satisfy $0\leq m_i\leq n_i$, replacing $m_i$ by $n_i-m_i$ changes the above sum to
\[
[T^{n-m}]P_{A,a}=\sum_{\substack{0 \leq m_i \leq n_i \\ \sum(n_i-m_i)=n-m}} \prod_i [T^{n_i-m_i}]P_{A_i,a_i}=\sum_{\substack{0 \leq m_i \leq n_i \\ \sum m_i=m}} \prod_j [T^{n_i-m_i}]P_{A_i,a_i}
\]
and multiplying by $(-1)^m=\prod_i (-1)^{m_i}$ gives the result.

For (\ref{sm prods::sm special case}), we know that for all $i\neq j$, we have $P_{A_i,0}(T)=T^{n_i}$ and so $s_{A_i,m_i}(a)=\delta_{m_i,0}$. Thus, by (\ref{sm prods::sm}) we see $s_{A,m}(a)=s_{A_j, m}(a_j)$. 
\end{proof}

\subsection{Galois closures}

We now define the Galois closure for a degree $n$ $R$-algebra $A$. For an element $a \in A$, let $a^{(i)}$ denote $1 \otimes \cdots \otimes a \otimes \cdots \otimes 1\in A^{\otimes n}$, where $a$ is in the $i$-th tensor factor. Then consider the left ideal $I_{A/R}$ in $A^{\otimes n}$ generated by the elements
\begin{equation}
e_j(a^{(1)}, a^{(2)}, \ldots, a^{(n)})-s_j(a) \label{eq:elemsymrelations}
\end{equation}
for every $a\in A$ and $1 \leq j \leq n$, where $e_j$ denotes the $j$-th elementary symmetric function and $s_j(a) = s_{A,j}(a) \cdot (1 \otimes \cdots \otimes 1)$.

\begin{definition}
\label{def:Galois closure}
The {\em Galois closure} of a degree $n$ $R$-algebra $A$ is defined to be the left $A^{\otimes n}$-module
\begin{equation} \label{eq:Galclosuredef}
G(A/R):=A^{\otimes n}/I_{A/R}.
\end{equation}
\end{definition}

\begin{remark}
\label{rmk:Galois closure agrees with commutative case}
If $A$ is a commutative ring of rank $n$ over $R$, and if we endow $A$ with the degree $n$ algebra structure via left multiplication as in Example \ref{ex:char-polyn-generic-polyn}, then it is immediate from the definition that the Galois closure $G(A/R)$ agrees with the $S_n$-closure introduced in \cite{gc}.
\end{remark}

\begin{remark}
\label{rmk:Galois closure commutative case ring structure}
Unlike the case of commutative rings considered in \cite{gc}, here $G(A/R)$ does not necessarily have a natural ring structure since $I_{A/R}$ is not necessarily a two-sided ideal. In fact, in many cases of interest (e.g., if $A$ is the ring of $n\times n$ matrices $\Mat_n(R)$ for $n\geq3$), if we were to replace $I_{A/R}$ by the two-sided ideal generated by the elements \eqref{eq:elemsymrelations}, the expression \eqref{eq:Galclosuredef} would become $0$.
\end{remark}

\begin{remark}
If $(A,R',\iota)$ is a degree $n$ algebra and $B$ is an $R$-algebra with a universally injective homomorphism $B \to A$, then $(B,R',\iota)$ inherits the structure of a degree $n$ algebra. Then since $I_{B/R} \subset I_{A/R}$, there is a well-defined homomorphism $G(B/R) \to G(A/R)$ of left $B^{\otimes n}$-modules.
\end{remark}

\begin{remark}
\label{rmk:Galois closure depends weakly on Rprime}
Let $(A,R',\iota)$ be a degree $n$ $R$-algebra and let $R''$ be a commutative $R'$-algebra. Let $\psi\colon\Mat_n(R')\to\Mat_n(R'')$ be the induced morphism. If $\psi\iota\colon A\to\Mat_n(R'')$ is universally injective, e.g., if $R' \to R''$ is universally injective, then $(A,R'',\psi\iota)$ is a degree $n$ $R$-algebra. In this case, the characteristic polynomials $P_{A,a}$ for $(A,R',\iota)$ and $(A,R'',\psi\iota)$ are the same, so the associated Galois closures agree. 
\end{remark}

By definition, the module $G(A/R)$ has $n$ distinct $A$-actions (one on each tensor factor). We denote the $i$-th action of $a \in A$ on an element $b \in G(A/R)$ as $a \cdot_i b$. Furthermore, the natural action of $S_n$ on $A^{\otimes n}$ induces an $S_n$-action on $G(A/R)$ with the following property: for all $\sigma\in S_n$, $a\in A$, and $b\in G(A/B)$, we have $\sigma(a\cdot_i b)=a\cdot_{\sigma(i)}\sigma(b)$. This gives $G(A/R)$ the structure of 
a left module over the twisted group ring $A^{\otimes n}*S_n$ (or equivalently, an $S_n$-equivariant left $A^{\otimes n}$-module).

\subsection{Key properties: base change and the product formula}  \label{sec:properties}
The focus of this subsection is to prove two main properties of Galois closures: they commute with base change and they satisfy a product formula. Our first step is to show that the left ideal $I_{A/R}$
is generated by the expressions \eqref{eq:elemsymrelations} for basis elements.
\begin{proposition}
\label{prop:I is generated by basis expressions}
Let $A$ be a degree $n$ $R$-algebra. If $u_1,\ldots,u_m$ is a basis for $A$ over $R$, then $I_{A/R}$ is the left $A^{\otimes n}$-ideal generated by the expressions
\[
e_j(a^{(1)}, a^{(2)}, \ldots, a^{(n)})-s_j(a)
\]
for $a\in\{u_1,\ldots,u_m\}$.
\end{proposition}
\begin{proof}
By definition, there is a finitely generated commutative $R$-algebra $R'$ and a universally injective $R$-algebra morphism $\iota\colon A\to\Mat_n(R')$ such that $P_{A,a}(T)=\det(T-\iota(a))\in R[T]$ for all $a\in A$. We then have
\begin{equation} \label{eq:1-aT}
\det(1-\iota(a)T)=\sum_{j=0}^n (-1)^js_{A,j}(a)T^j.
\end{equation}
As shown in \cite[Lemma 11]{gc}, for a noncommutative polynomial ring $\ZZ\langle X,Y \rangle$ over $\ZZ$ generated by $X$ and $Y$, there is a unique sequence $\{f_d(X,Y)\}_{d=0}^\infty$ of homogeneous degree $d$ polynomials in $\ZZ\langle X,Y \rangle$ such that 
$$(1 - (X+Y)T) = (1-XT)(1-YT)\prod_{d=0}^\infty (1 - f_d(X,Y) X Y T^{d+2})$$
in $\ZZ\langle X,Y \rangle[[T]]$. 
In particular, for any $a,b\in A$, letting $x=\iota(a)$ and $y=\iota(b)$, we have
\begin{equation} \label{eq:XYT}
1-(x+y)T=(1-xT)(1-yT)\prod_{d=0}^{k-2}(1-f_d(x,y)xyT^{d+2})\mod T^{k+1}.
\end{equation}
Taking determinants of both sides of \eqref{eq:XYT} and equating the coefficients of $T^k$, equation \eqref{eq:1-aT} yields an expression for $s_{A,k}(a+b)$ in terms of $s_{A,k}(a)$, $s_{A,k}(b)$, and $s_{A,i}(q_j(a,b))$ for $i<k$, where the $q_j$ are noncommutative polynomials. In the case where $A$ is the split degree $n$ algebra $R^n$ as in Example \ref{ex:split-alg}, the $s_{A,k}$ are the elementary symmetric functions $e_k$.

Combining the above with the observation that $s_{A,i}(ra)=r^is_{A,i}(a)$ for all $r\in R$, we see by induction on $k$ that $s_{A,k}(a)$ is expressible in terms of the $s_{A,j}(u_\ell)$ for $j\leq k$. Since the elementary symmetric functions $e_k$ satisfy these same relations (since they correspond to the special case where $A=R^n$), we conclude that $I_{A/R}$ is generated by the expressions $e_j(a^{(1)}, a^{(2)}, \ldots, a^{(n)})-s_{j}(a)$ for $a\in\{u_1,\ldots,u_m\}$.
\end{proof}

Next, we show that if $A$ is a degree $n$ $R$-algebra, and $R\to S$ is a map of commutative rings, then $A\otimes_R S$ carries a natural degree $n$ $S$-algebra structure.

\begin{lemma}
\label{l:base changing the degree n alg structure}
Let $R$ be a commutative ring and $S$ a commutative $R$-algebra. If $(A,R',\iota)$ is a degree $n$ $R$-algebra, then $(A\otimes_R S,R'\otimes_R S,\iota\otimes1)$ is a degree $n$ $S$-algebra.
\end{lemma}
\begin{proof}
By definition, $\iota$ is universally injective, so $\iota\otimes1\colon A\otimes_RS\to\Mat_n(R'\otimes_RS)$ is as well. It remains to prove that if $b\in A\otimes_RS$, then the characteristic polynomial of $(\iota\otimes1)(b)$ lives in $S[T]$. The proof of Proposition \ref{prop:I is generated by basis expressions} yields an integral polynomial expression for the characteristic polynomial of $x+y$ in terms of the characteristic polynomials of $x$ and $y$. Since $b$ is a sum of pure tensors, we are therefore reduced to the case where $b$ is a pure tensor itself, e.g., $b=a\otimes c$ for $a \in A$ and $c \in S$. Then $(\iota\otimes1)(b)$ is the product of $\iota(a)\otimes1$ and the scalar $c\in S$. So $s_{A\otimes_RS,i}(b)=c^i s_{A,i}(a)$. Since $(A,R',\iota)$ is a degree $n$ $R$-algebra, we have $s_{A,i}(a)\in R$, and hence $s_{A\otimes_RS,i}(b)\in S$, as desired.
\end{proof}
\begin{remark}
\label{rmk:base changed sj of pure tensors}
While proving Lemma \ref{l:base changing the degree n alg structure}, we showed that $s_{A\otimes_R S,i}(a\otimes c)=c^i s_{A,i}(a)$ for all $a\in A$ and $c\in S$.
\end{remark}

The map $A\to A\otimes_R S$ induces a map $A^{\otimes n}\to (A\otimes_R S)^{\otimes n}$. This latter morphism sends $I_{A/R}$ into $I_{(A\otimes_R S)/S}$ by Remark \ref{rmk:base changed sj of pure tensors}, and hence induces a map $G(A/R)\otimes_R S\to G((A\otimes_R S)/S)$, which we refer to as the {\em base change map}. Note that $(A \otimes_R S)^{\otimes n} \simeq A^{\otimes n} \otimes_R S$ acts on both $G(A/R) \otimes_R S$ and $G((A \otimes_R S)/S)$.

\begin{theorem}[Base change]
\label{thm:base change}
Let $A$ be a degree $n$ $R$-algebra and let $S$ be a commutative $R$-algebra. The base change map yields an isomorphism
\[
G(A/R)\otimes_RS\simeq G((A\otimes_R S)/S)
\]
of left modules over $(A \otimes_R S)^{\otimes n}*S_n$.
\end{theorem}
\begin{proof}
Let $\epsilon_j(a):=e_j(a^{(1)}, a^{(2)}, \ldots, a^{(n)})-s_j(a)$. If $u_1,\ldots,u_m$ is a basis for $A$ over $R$, then the $u_\ell\otimes 1$ give a basis for $A\otimes_R S$ over $S$. Proposition \ref{prop:I is generated by basis expressions} implies that $I_{A/R}$ is generated by the expressions $\epsilon_j(u_\ell)$ and $I_{(A\otimes_R S)/S}$ is generated by the expressions $\epsilon_j(u_\ell\otimes1)=\epsilon_j(u_\ell)\otimes1$. Hence $I_{(A\otimes_R S)/S}$ is the extension of the ideal $I_{A/R}$. Consequently, the base change map is an isomorphism $G(A/R)\otimes_RS\simeq G((A\otimes_R S)/S)$.
\end{proof}

We next compute the Galois closure of a product in terms of the Galois closures of the factors.
\begin{theorem}[Product formula]
\label{thm:product-formula}
For $1 \leq j \leq k$, let $A_j$ be a degree $n_j$ $R$-algebra and endow $A = A_1 \times \cdots \times A_k$ with the associated structure of a degree $n = \sum_{i=1}^k n_i$ $R$-algebra. Then we have an isomorphism of left $(A^{\otimes n}*S_n)$-modules
\begin{equation} \label{eq:productformula}
G(A/R) \simeq \left( G(A_1/R)\otimes\cdots\otimes G(A_k/R) \right)^N
\end{equation}
where $N$ is the multinomial coefficient ${n\choose n_1,\ldots,n_k}$. As a representation of $S_n$, the right hand side of \eqref{eq:productformula} is the induced representation from $S_{n_1} \times \cdots \times S_{n_k}$ acting on $G(A_1/R) \otimes \cdots \otimes G(A_k/R)$, and the $A^{\otimes n}$-action on the right hand side of \eqref{eq:productformula} is given by $n_i$ actions of $A$ on each $G(A_i/R)$.
\end{theorem}

\begin{proof}
Since $A=\prod_j A_j$, there exist idempotents $\varepsilon_j$ in the center $Z(A)$ of $A$ for $1 \leq j \leq k$, such that $A_j=\varepsilon_jA=A\varepsilon_j$ and $\varepsilon_i\varepsilon_j=\delta_{i,j}\varepsilon_j$ where $\delta$ is the Kronecker delta function. Let $[k]=\{1,2,\cdots,k\}$ and for every $n$-tuple $\underline{i}=(i_1,\cdots,i_n)\in [k]^n$, let $A_{\underline{i}}:=A_{i_1}\otimes A_{i_2}\otimes\cdots\otimes A_{i_n}$. Then
\[
A^{\otimes n}=\prod_{\underline{i}\in [k]^n}A_{\underline{i}}.
\]
This product decomposition corresponds to the idempotents $\varepsilon_{\underline{i}}\in Z(A^{\otimes n})$ defined by
\[
\varepsilon_{\underline{i}}=\varepsilon_{i_1}\otimes \varepsilon_{i_2}\otimes\cdots \otimes \varepsilon_{i_n}=\varepsilon_{i_1}^{(1)} \varepsilon_{i_2}^{(2)}\cdots \varepsilon_{i_n}^{(n)}.
\]
Notice that if $J\subseteq A^{\otimes n}$ is a left ideal, then $\varepsilon_{\underline{i}}J=J\varepsilon_{\underline{i}}$ is a left ideal, which can be identified with a left ideal $J_{\underline{i}}$ of $A_{\underline{i}}$. Moreover, $J=\prod_{\underline{i}} J_{\underline{i}}$. In particular,
\[
I_{A^{\otimes n}/R}=\prod_{\underline{i}\in[k]^n} I_{\underline{i}}.
\]

Our first goal is to show that $I_{\underline{i}}=A_{\underline{i}}$ unless $\#\{\ell\mid i_\ell=j\}=n_j$ for every $j$, i.e., unless $A_{\underline{i}}\simeq A_{i_1}^{\otimes n_1}\otimes A_{i_2}^{\otimes n_2}\otimes\cdots A_{i_k}^{\otimes n_k}$. 
Let $\underline{i}$ be an $n$-tuple for which this does not hold; then there is some $j$ with $\#\{\ell\mid i_\ell=j\}<n_j$. By Lemma \ref{l:sm for products} (\ref{sm prods::sm special case}), we know that $s_{A,n_j}(\varepsilon_j)=s_{A_j,n_j}(1)=1$. So, the $n_j$-th elementary symmetric function in the $\varepsilon_j^{(\ell)}$ equals 1 in the module $G(A/R)$, i.e.,
\[
1-\sum_{1\leq r_1<r_2<\cdots<r_{n_j}\leq n}\varepsilon_j^{(r_1)}\varepsilon_j^{(r_2)}\cdots \varepsilon_j^{(r_{n_j})}\in I_{A^{\otimes n}/R}.
\]
Since $j$ occurs fewer than $n_j$ times in the $n$-tuple $\underline{i}$, we see $\varepsilon_{\underline{i}}\varepsilon_j^{(r_1)}\varepsilon_j^{(r_2)}\cdots \varepsilon_j^{(r_{n_j})}=0$ for every summand above, and so $\varepsilon_{\underline{i}}\in \varepsilon_{\underline{i}} I_{A^{\otimes n}/R}=I_{\underline{i}}$, i.e., we have $I_{\underline{i}}=A_{\underline{i}}$, as desired.

Next, let $\underline{i}\in [k]^n$ such that $A_{\underline{i}}\simeq A_{i_1}^{\otimes n_1}\otimes A_{i_2}^{\otimes n_2}\otimes\cdots A_{i_k}^{\otimes n_k}$. We show in this case that $A_{\underline{i}}/I_{\underline{i}}\simeq G(A_1/R)\otimes G(A_2/R)\otimes\cdots G(A_k/R)$. To do so, it is enough to consider the specific case where $\underline{i}=(i_1,\cdots,i_n)$ is equal to $(\underbrace{1,1,\cdots,1}_{n_1},\cdots,\underbrace{k,k,\cdots,k}_{n_k})$, since this is the case up to permutation. First note that $A$ is generated by elements of the form $a_j\varepsilon_j$ where $a_j\in A_j$ and $1\leq j\leq n$. So, by Theorem \ref{thm:base change}, we know that $I_{A^{\otimes n}/R}$ is generated as a left ideal by elements of the form
\[
s_{A,m}(a_j\varepsilon_j)-\sum_{1\leq r_1<r_2<\cdots <r_m\leq n}(a_j\varepsilon_j)^{(r_1)}(a_j\varepsilon_j)^{(r_2)}\cdots (a_j\varepsilon_j)^{(r_m)},
\]
and so $I_{\underline{i}}$ is generated by the above elements after left multiplying by $\varepsilon_{\underline{i}}$. First notice that by Lemma \ref{l:sm for products} (\ref{sm prods::sm special case}) we know $s_{A,m}(a_j\varepsilon_j)=s_{A_j,m}(a_j)$, which is 0 if $m>n_j$. Next note that $i_\ell=j$ if and only if $n_1+n_2+\cdots+n_{j-1}+1\leq\ell\leq n_1+n_2+\cdots+n_j$. So, multiplying the above expression by $\varepsilon_{\underline{i}}$ is 0 if $m>n_j$, and otherwise we obtain
\[
s_{A_j,m}(a_j)-\sum_{n_1+\cdots+n_{j-1}+1\leq r_1<r_2<\cdots <r_m\leq n_1+\cdots+n_j}\varepsilon_{\underline{i}}\, a_j^{(r_1)}a_j^{(r_2)}\cdots a_j^{(r_m)}
\]
which is nothing more than
\[
s_{A_j,m}(a_j)-\underbrace{\varepsilon_1\otimes\cdots \varepsilon_1}_{n_1} \otimes\cdots\otimes\sum_{1\leq r_1<r_2<\cdots <r_m\leq n_j} a_j^{(r_1)}a_j^{(r_2)}\cdots a_j^{(r_m)}\otimes\cdots\otimes \underbrace{\varepsilon_k\otimes\cdots \varepsilon_k}_{n_{k}}.
\]
This shows $I_{\underline{i}}=I_{A_1^{\otimes n_1}/R}\otimes\cdots\otimes I_{A_k^{\otimes n_k}/R}$ and therefore $A_{\underline{i}}/I_{\underline{i}}\simeq G(A_1/R)\otimes G(A_2/R)\otimes\cdots G(A_k/R)$ as desired.
\end{proof}

\section{Examples of Galois closures} \label{sec:examples}

In this section, we give many examples of Galois closures.

\subsection{Trivial algebras} \label{sec:extrivalgs}
As in Example \ref{ex:triv-generic-polyn}, for any $n \geq 1$, we may view the $R$-algebra $A = R$ as a degree $n$ algebra with characteristic polynomial $P_{R,r}(T) = (T-r)^n$. It is easy to check that $I_{A/R}=0$, so the Galois closure $G(A/R)$ is isomorphic to $R$ itself. The $S_n$-action is trivial, and the $n$ $A$-actions are all the same, namely the usual multiplication by elements of $A = R$. This seemingly trivial example plays a role in many of the examples in Section \ref{sec:vinberg}.

\subsection{Quadratic algebras} \label{sec:exquadalgs}
\label{subsec:quadratic algebras}

Suppose $A$ is a degree $2$ $R$-algebra.  Then every element $a \in A$ satisfies an equation of the form
	\begin{equation} \label{eq:monicpolyfordeg2}
	a^2 - \Tr(a) a + \No(a) = 0,
	\end{equation}
where $\Tr(a)$ and $\No(a)$ are the trace and norm, respectively, of $a$. Let $\overline{a} := \Tr(a) - a$, which one should think of as the conjugate of $a$. It is easy to check that $a \overline{a} = \overline{a} a = \No(a)$ and $\overline{ab} = \overline{b} \overline{a}$.

The Galois closure $G(A/R)$ is the quotient of $A \otimes A$ by the left ideal $I_{A/R}$, generated by the elements
$$a \otimes 1 + 1 \otimes a - \Tr(a) (1 \otimes 1) \quad \mathrm{and} \quad a \otimes a - \No(a) (1 \otimes 1)$$
for all $a \in A$.  We will show that $G(A/R)$ is isomorphic to $A$ in this case. We first give $A$ the structure of a left $(A^{\otimes2}*S_2)$-module as follows: given $b \in A$ and a pure tensor $a_1\otimes a_2\in A^{\otimes2}$, let $(a_1\otimes a_2)\cdot b:=a_1b\overline{a}_2$. If $\sigma\in S_2$ denotes the non-trivial element, then the $S_2$-action on $A$ is given by $\sigma(b):=\overline{b}$.

\begin{proposition} \label{prop:quadratic}
If $A$ is a degree 2 $R$-algebra, then the morphism  $\varphi \colon A\otimes A \to A$ given by $\varphi(b\otimes c)=b\overline{c}$ induces an isomorphism
\[
G(A/R)\simeq A
\]
of left $(A^{\otimes 2}*S_2)$-modules.
\end{proposition}
\begin{proof}
One easily checks that $\varphi$ is well defined and a morphism of left $(A^{\otimes 2}*S_2)$-modules.  It is clear that $\varphi(I_{A/R})=0$, so we obtain an induced map $\overline\varphi \colon G(A/R)\to A$. Note that $\varphi(a \otimes 1)=a$ for all $a\in A$, so $\overline\varphi$ is surjective.

Since $\overline\varphi$ is a morphism of left $(A^{\otimes 2}*S_2)$-modules, to complete the proof, it is enough to show $\overline\varphi$ is an isomorphism of $R$-modules. Consider the $R$-module morphism $\psi\colon A\to G(A/R)$ given by $\psi(a)=a^{(1)}=a\otimes1$. It is injective since $\overline\varphi\psi(a)=a$. It is surjective because 
\[
b \otimes c = b^{(1)}c^{(2)}=b\cdot_1 c^{(2)}=b\cdot_1(\Tr(c) (1 \otimes 1) - c^{(1)})=\Tr(c) b^{(1)}-(bc)^{(1)}
\]
for all $b, c \in A$, i.e., $G(A/R)$ is generated as an $R$-module by the image of elements of the form $a^{(1)}$ for $a \in A$.
\end{proof}

\begin{remark}
\label{rmk:quadratic alg already Galois}
Proposition \ref{prop:quadratic} generalizes the fact that a separable degree 2 field extension $L/K$ is already Galois and hence its Galois closure is $L$. Note that in the case of quadratic $R$-algebras, since $G(A/R)\simeq A$, the Galois closure inherits a ring structure.
\end{remark}

\subsection{Cubic algebras built from smaller degree algebras} \label{sec:cubicfromsmaller}
We give some examples of $G(A/R)$ where $A$ has degree $3$ but is the product of smaller degree algebras. These Galois closures may be easily computed using the product formula (Theorem \ref{thm:product-formula}), but in this section, we show how to work with them explicitly.

The simplest case of a decomposable degree $3$ $R$-algebra is $A = R \times R \times R$.  Each element $a = (r_1, r_2, r_3) \in A$ satisfies the polynomial
\begin{equation} \label{eq:deg3poly}
a^3 - t a^2 + s a - n 1_A = 0,
\end{equation}
where the {\em trace} $\Tr(a)$ is $s_{A,1}(a) = t = r_1 + r_2 + r_3$, the {\em spur} $\Spur(a)$ is $s_{A,2}(a) = s = r_1 r_2 + r_1 r_3 + r_2 r_3$, the {\em norm} $\No(a)$ is $s_{A,3}(a) = n = r_1 r_2 r_3$, and $1_A$ denotes the multiplicative identity element $(1,1,1)$ in $A$.

We claim that $G(A/R)$ is isomorphic to $R^{\oplus 6}$ as left $(A^{\otimes 3}*S_3)$-modules, where the left $(A^{\otimes 3}*S_3)$-module structure on $R^{\oplus 6}$ is given as follows. We index each of the six copies of $R$ in $R^{\oplus 6}$ by the six permutations of $\{1,2,3\}$.  The three actions of $(r_1, r_2, r_3) \in A$ on $(c_{ijk})_{\{i,j,k\}=\{1,2,3\}} \in R^{\oplus 6}$ are as follows:
\begin{align*}
(r_1,r_2,r_3) \cdot_1 (c_{ijk})_{\{i,j,k\}=\{1,2,3\}} &:= (r_i c_{ijk})_{\{i,j,k\}=\{1,2,3\}}\\
(r_1,r_2,r_3) \cdot_2 (c_{ijk})_{\{i,j,k\}=\{1,2,3\}} &:= (r_j c_{ijk})_{\{i,j,k\}=\{1,2,3\}}\\
(r_1,r_2,r_3) \cdot_3 (c_{ijk})_{\{i,j,k\}=\{1,2,3\}} &:= (r_k c_{ijk})_{\{i,j,k\}=\{1,2,3\}}.
\end{align*}
The action of $\sigma \in S_3$ on $R^{\oplus 6}$ is the standard action on the indices, i.e., $\sigma((c_{ijk})_{ijk}) = (c_{\sigma(i),\sigma(j),\sigma(k)})_{ijk}$.  Then we define the morphism $\varphi\colon A^{\otimes 3} \to R^{\oplus 6}$ of left $A^{\otimes 3}$-modules by linearly extending
$$\varphi( (b_1,b_2,b_3) \otimes (c_1,c_2,c_3) \otimes (d_1,d_2,d_3) ) = (b_i c_j d_k)_{\{i,j,k\}=\{1,2,3\}}.$$
It is easy to check that the image of $I_{A/R}$ is $0$ in $R^{\oplus 6}$, so we obtain an induced map $\overline{\varphi} \colon G(A/R) \to R^{\oplus 6}$.

\begin{proposition}
The map $\overline{\varphi}$ is an isomorphism of left $(A^{\otimes 3}*S_3)$-modules:
\[
G(A/R) \simeq R^{\oplus 6}.
\]
\end{proposition}
\begin{proof}
This is a special case of Proposition \ref{prop:k-B} below, with $B = R \times R$.
\end{proof}

We next consider the more general situation where $B$ is a quadratic $R$-algebra and let $A = R \times B$. Recall from \S \ref{sec:exquadalgs} that there is a trace form $\Tr_B$ and norm form $\No_B$ on $B$ such that $b \in B$ satisfies the quadratic polynomial \eqref{eq:monicpolyfordeg2}: $b^2 - \Tr_B(b) b + \No_B(b)=0$. 
Then by the definition of the product polynomial $P_{R\times B,(r,b)}$, an element $a = (r, b) \in A$ satisfies the cubic polynomial \eqref{eq:deg3poly}, where $t = \Tr(a) = r + \Tr_B(b)$, $s= \Spur(a) = r \Tr_B(b) + \No_B(b)$, $n = \No(a) = r \No_B(b)$, and $1_A = (1,1_B)$.  Recall as well from \S\ref{subsec:quadratic algebras} that for $b\in B$, we define $\overline{b}=\Tr_B(b)-b$.

We claim that the Galois closure $G(A/R)$ is isomorphic to $B^{\oplus 3}$ where we endow $B^{\oplus 3}$ with a left $(A^{\otimes 3}*S_3)$-module structure as follows. The three actions of $(r,c) \in A$ on $(b_1, b_2, b_3) \in B^{\oplus 3}$ are given by
\begin{align*}
(r,c)\cdot_1(b_1,b_2,b_3) &:=(r b_1, c b_2,b_3\overline{c}) \\
(r,c)\cdot_2(b_1,b_2,b_3) &:=(b_1\overline{c},r b_2,c b_3) \\
(r,c)\cdot_3(b_1,b_2,b_3) &:=(c b_1,b_2\overline{c},r b_3).
\end{align*}
The $S_3$-action on $B^{\oplus 3}$ is given by
\[
\sigma(b_1,b_2,b_3):=
\begin{cases}
(b_{\sigma^{-1}(1)},b_{\sigma^{-1}(2)},b_{\sigma^{-1}(3)}) & \textrm{if } \sgn(\sigma)=1\\
(\overline{b}_{\sigma^{-1}(1)},\overline{b}_{\sigma^{-1}(2)},\overline{b}_{\sigma^{-1}(3)}) & \textrm{if } \sgn(\sigma)=-1.
\end{cases}
\]
Then we may define a morphism $\varphi \colon A^{\otimes 3}\to B^{\oplus 3}$ of left $(A^{\otimes 3}*S_3)$-modules by linearly extending
\[
\varphi((r_1,b_1)\otimes(r_2,b_2)\otimes(r_3,b_3))=(r_1 b_3 \overline{b}_2,r_2 b_1\overline{b}_3, r_3 b_2\overline{b}_1).
\]
An easy check shows that $\varphi(I_{A/R})=0$, so we obtain an induced map $\overline\varphi \colon G(A/R) \to B^{\oplus 3}$.  We then have
\begin{proposition}
\label{prop:k-B}
The map $\overline\varphi$ is an isomorphism of left $(A^{\otimes 3}*S_3)$-modules:
\[
G(A/R) \simeq B^{\oplus 3}.
\]
\end{proposition}
\begin{proof}
We see that $\overline\varphi$ is surjective since $\varphi$ is: for every $b\in B$, we have
\begin{gather*}
\varphi((0,1)^{(2)}(0,b)^{(3)})=(b,0,0) \\
\varphi((0,\overline{b})^{(3)}-(0,1)^{(2)}(0,\overline{b})^{(3)})=(0,b,0) \\
\varphi((0,b)^{(2)}-(0,1)^{(2)}(0,\overline{b})^{(3)})=(0,0,b).
\end{gather*}

Since we know $\overline{\varphi}$ is a surjective map of left $(A^{\otimes 3}*S_3)$-modules, to prove it is an isomorphism, it is enough to show it is an isomorphism of $R$-modules. To do so, we need only find an $R$-module map $\psi\colon B^{\oplus 3}\to G(A/R)$ which is a surjective section of $\overline\varphi$. Then we can define $\psi$ by $\psi(b,0,0)=(0,1)^{(2)}(0,b)^{(3)}$, $\psi(0,b,0)=(0,\overline{b})^{(3)}-(0,1)^{(2)}(0,\overline{b})^{(3)}$, and $\psi(0,0,b)=(0,b)^{(2)}-(0,1)^{(2)}(0,\overline{b})^{(3)}$. It then suffices to show that $G(A/R)$ is generated as an $R$-module by elements of the form $(0,b)^{(2)}$, $(0,b)^{(3)}$, and $(0,1)^{(2)}(0,b)^{(3)}$.

For any element $\alpha \in A$, the trace relation $\alpha^{(1)} + \alpha^{(2)} + \alpha^{(3)} - \Tr(\alpha) \in I_{A/R}$ trivially shows that any element $\alpha^{(1)} = \alpha \otimes 1 \otimes 1 \in A^{\otimes 3}$ may be written in $G(A/R)$ as an $R$-linear combination of $1_A$, $\alpha^{(2)}$, and $\alpha^{(3)}$. In particular, any element of $G(A/R)$ may thus be written as a linear combination of elements of the form $\beta^{(2)} \gamma^{(3)}$ for $\beta, \gamma \in A$. Also, for $\beta = (r,b) \in R \times B = A$, we see that $\beta^{(i)} = r 1_A + (0,b-r)^{(i)}$, so the Galois closure $G(A/R)$ is generated as an $R$-module by $1_A$ and elements of the form $(0,b)^{(2)}$, $(0,b)^{(3)}$, and $(0,b)^{(2)} (0,b')^{(3)}$ for $b, b' \in B$.

We claim that for $b, b' \in B$, we can write $(0,b)^{(2)} (0,b')^{(3)}$ in the form $(0,1)^{(2)}(0,b'')^{(3)}$ with $b''\in B$; specifically, $b''=b'\overline{b}$. 
First, expanding the equation $(b+1)^2-\Tr_B(b+1)(b+1)-\No_B(b+1)=0$ yields
\begin{equation} \label{eq:Nb1vsNb}
\No_B(b+1) - \No_B(b) -1 = \Tr_B(b).
\end{equation}
Next, the trace and spur relations tell us
\begin{align*}
(0,b)^{(2)}(0,b)^{(3)}&=\No_B(b)-(0,b)^{(1)}(0,b)^{(2)}-(0,b)^{(1)}(0,b)^{(3)} = \No_B(b)+(0,b)\cdot_1((0,b)^{(1)}-\Tr_B(b))\\
&=\No_B(b) + (0,b^2-\Tr_B(b)b)^{(1)}=\No_B(b) - (0,\No_B(b))^{(1)}=\No_B(b)(1_A-(0,1)^{(1)})\\
&=\No_B(b)(-1_A+(0,1)^{(2)}+(0,1)^{(3)}),
\end{align*}
where the last equality holds because $\Tr_A(0,1)=\Tr_B(1_B)=2$. Applying this equation once to lefthand side and twice to the righthand side of
$$(0,b+1)^{(2)} (0,b+1)^{(3)} = (0,b)^{(2)} (0,b)^{(3)} + (0,b)^{(2)} (0,1)^{(3)} + (0,1)^{(2)} (0,b)^{(3)} + (0,1)^{(2)} (0,1)^{(3)}$$
and making use of \eqref{eq:Nb1vsNb} implies
\begin{equation} \label{eq:0b01}
(0,b)^{(2)} (0,1)^{(3)} + (0,1)^{(2)} (0,b)^{(3)} = - \Tr_B(b) (1_A + (0,1)^{(2)} + (0,1)^{(3)})
\end{equation}
which simplifies to
$$(0,b)^{(2)} (0,1)^{(3)} = -  \Tr_B(b) (1 \otimes (1,0) \otimes (1,0)) + (0,1)^{(2)} (0,\overline{b})^{(3)}.$$
Acting on the left by $(0,b')^{(3)}$ shows that $(0,b)^{(2)} (0,b')^{(3)} = (0,1)^{(2)} (0,b'\overline{b})^{(3)}$, i.e., we have shown that $(0,b)^{(2)} (0,b')^{(3)}$ is of the form $(0,1)^{(2)}(0,b'')^{(3)}$.

So far, we have shown that $G(A/R)$ is generated as an $R$-module by $1_A$ and elements of the form $(0,b)^{(2)}$, $(0,b)^{(3)}$, and $(0,1)^{(2)} (0,b')^{(3)}$. It remains to remove $1_A$ from our generating set. We have shown above that $(0,b)^{(2)}(0,b)^{(3)}=\No_B(b)(-1_A+(0,1)^{(2)}+(0,1)^{(3)})$. Substituting in $b = 1$ shows that $(0,1)^{(2)}(0,1)^{(3)} = -1_A + (0,1)^{(2)} + (0,1)^{(3)}$. Therefore, $1_A$ is also a linear combination of elements of the form $(0,b)^{(2)}$, $(0,b)^{(3)}$, and $(0,1)^{(2)}(0,b)^{(3)}$. This concludes the proof.
\end{proof}

\begin{remark}
\label{rmk:decomposable cubic ring str}
As a consequence of Proposition \ref{prop:k-B}, we see that $G(A/R)$ inherits a ring structure when $A=R\times B$ with $B$ a degree 2 $R$-algebra. This is not true more generally for indecomposable degree 3 $R$-algebras, e.g., for the case $A=\Mat_3(R)$ considered in \S\ref{subsec:endo-rings}.
\end{remark}

\subsection{Endomorphism rings (or matrix algebras)}
\label{subsec:endo-rings}
Let $V$ be a free rank $n$ module over $R$ and let $A=\End(V)$ be the ring of $R$-module endomorphisms. Then as in Example \ref{ex:matrix-alg-polyn}, we may view $A$ as a degree $n$ $R$-algebra where for each $\alpha\in A$, the polynomial $P_{A,\alpha}(T)$ is the characteristic polynomial of $\alpha$ viewed as an endomorphism of $V$; the trace and the norm of an endomorphism coincide with their usual definitions. 

Via the canonical isomorphism $A\simeq V\otimes V^*$, we have an isomorphism $A^{\otimes n} \simeq V^{\otimes n}\otimes (V^*)^{\otimes n}$. The natural left $(A^{\otimes n}*S_n)$-module structure on $A^{\otimes n}$ then induces such a structure on $V^{\otimes n}\otimes (V^*)^{\otimes n}$. Explicitly, $\sigma\in S_n$ acts on the pure tensors via
\[
\sigma(v_{1}\otimes\cdots\otimes v_{n}\otimes f_{1}\otimes\cdots\otimes f_{n}):=v_{{\sigma^{-1}(1)}}\otimes\cdots\otimes v_{{\sigma^{-1}(n)}}\otimes f_{{\sigma^{-1}(1)}}\otimes\cdots\otimes f_{{\sigma^{-1}(n)}}
\]
where $v_j \in V$ and $f_j \in V^*$. The $i$-th action of $\alpha \in A$ is given by acting on the $i$-th factor of $V$:
\[
\alpha\cdot_i(v_{1}\otimes\cdots\otimes v_{n}\otimes f_{1}\otimes\cdots\otimes f_{n}):=v_{1}\otimes\cdots\otimes\alpha(v_i)\otimes\cdots\otimes v_{n}\otimes f_{1}\otimes\cdots\otimes f_{n}.
\]
We next define a left $(A^{\otimes n}*S_n)$-module structure on $V^{\otimes n}\otimes \extp^n(V^*)$ as follows. For $\sigma\in S_n$, let 
\[
\sigma(v_{1}\otimes\cdots\otimes v_{n}\otimes (f_{1}\wedge\cdots\wedge f_{n})):=v_{{\sigma^{-1}(1)}}\otimes\cdots\otimes v_{{\sigma^{-1}(n)}}\otimes (f_{{\sigma^{-1}(1)}}\wedge\cdots\wedge f_{{\sigma^{-1}(n)}});
\]
in other words, for a form $\omega \in \extp^n(V^*)$, we have
\[
\sigma(v_{1}\otimes\cdots\otimes v_{n}\otimes \omega):=\sgn(\sigma)v_{{\sigma^{-1}(1)}}\otimes\cdots\otimes v_{{\sigma^{-1}(n)}}\otimes \omega.
\]
For $\alpha\in A$, the $i$-th $A$-action is given by
\[
\alpha\cdot_i(v_{1}\otimes\cdots\otimes v_{n}\otimes \omega):=v_{1}\otimes\cdots\otimes\alpha(v_i)\otimes\cdots\otimes v_{n}\otimes \omega.
\]

Let $\pi\colon (V^*)^{\otimes n}\to\extp^n(V^*)$ be the natural map $\pi(f_1\otimes\cdots\otimes f_n)=f_1\wedge\cdots\wedge f_n$. We then obtain a map of left $(A^{\otimes n}*S_n)$-modules. 
\begin{equation} \label{eq:phimatrixalg}
\varphi \colon A^{\otimes n}=V^{\otimes n}\otimes (V^*)^{\otimes n}\stackrel{\id\otimes\pi}{\longrightarrow}V^{\otimes n}\otimes \extp^n(V^*).
\end{equation}
Explicitly, if $u_1,\ldots,u_n$ is a basis of $V$, then for $\alpha_1, \ldots, \alpha_n \in A = \End(V)$, we compute
\begin{equation} \label{eq:phionendtensor}
\varphi(\alpha_1\otimes\cdots\otimes\alpha_n)=\sum_{\sigma\in S_n}\sgn(\sigma) \alpha_1(u_{\sigma(1)})\otimes\cdots\otimes\alpha_n(u_{\sigma(n)})\otimes(u_1^*\wedge\cdots\wedge u_n^*).
\end{equation}
We show that $\varphi$ induces an isomorphism between $G(A/R)$ and $V^{\otimes n}\otimes \extp^n(V^*)$. First, we show that we may find a fairly simple basis for $A$.

\begin{lemma}
\label{l:basis of diagonalizable matrices}
Let $V$ be a free rank $n$ $R$-module. Then the endomorphism ring $\End(V)$ has a basis $\beta_1,\ldots,\beta_{n^2}$ over $R$ where each $\beta_i$ is a linear operator that is diagonalizable over $R$.
\end{lemma}
\begin{proof}
After choosing an $R$-basis for $V$, we have an isomorphism $\End(V)\simeq \Mat_n(R)$. Let $e_{ij}$ denote the matrix whose entries are all $0$ except for a $1$ in the $(i,j)$-th position. Then for $i \neq j$, the matrix $e_{ij} + e_{ii}$ is diagonalizable: if $P = e_{ii} + e_{ij} - e_{jj} + \sum_{k \neq i,j} e_{kk}$, then $P^{-1} = P$ and $P^{-1} (e_{ij} + e_{ii})P = e_{ii}$. So, we can choose our desired basis to be the $n$ elements $e_{11},\ldots,e_{nn}$ as well as the $n^2-n$ elements $e_{ij} + e_{ii}$ and $e_{ji} + e_{jj}$ with $i<j$.
\end{proof}

\begin{theorem}
\label{thm:gcforEndV}
Let $V$ be a free rank $n$ $R$-module and $A = \End(V)$. The map $\varphi$ of \eqref{eq:phimatrixalg} induces an isomorphism
\[
G(A/R)\simeq V^{\otimes n}\otimes \extp^n(V^*)
\]
of left $(A^{\otimes n}*S_n)$-modules.
\end{theorem}
\begin{proof}
We begin by showing that $\varphi(I_{A/R})=0$. That is, we show that for all $\alpha\in A$,
\[
\varphi(\sum_{i_1<\cdots<i_r}\alpha^{(i_1)}\cdots\alpha^{(i_r)})=\varphi(s_{A,r}(\alpha)).
\]
By Proposition \ref{prop:I is generated by basis expressions}, it suffices to prove this as $\alpha$ ranges over a basis of $A$ over $R$. Combining this with Lemma \ref{l:basis of diagonalizable matrices}, we may assume that $\alpha$ is diagonalizable over $R$. Choose a basis $u_1,\ldots, u_n$ of $V$ such that $\alpha(u_i)=\lambda_i u_i$ for $\lambda_i \in R$. Let $\omega=u^*_1\wedge\cdots\wedge u^*_n$. From \eqref{eq:phionendtensor}, we see
\[
\varphi(\alpha^{(i_1)}\cdots\alpha^{(i_r)})=\sum_{\sigma\in S_n}\lambda_{\sigma(i_1)}\cdots\lambda_{\sigma(i_r)}\sgn(\sigma)u_{\sigma(1)}\otimes\cdots\otimes u_{\sigma(n)}\otimes\omega.
\]
Since 
\[
s_{A,r}(\alpha)=\sum_{i_1<\cdots<i_r}\lambda_{i_1}\cdots\lambda_{i_r},
\]
we obtain
\begin{align*}
\varphi(\sum_{i_1<\cdots<i_r}\alpha^{(i_1)}\cdots\alpha^{(i_r)})&=\sum_{\sigma\in S_n}\sum_{i_1<\cdots<i_r}\lambda_{\sigma(i_1)}\cdots\lambda_{\sigma(i_r)}\sgn(\sigma)u_{\sigma(1)}\otimes\cdots\otimes u_{\sigma(n)}\otimes\omega \\
&=\sum_{\sigma\in S_n}\sum_{i_1<\cdots<i_r}\lambda_{i_1}\cdots\lambda_{i_r}\sgn(\sigma)u_{\sigma(1)}\otimes\cdots\otimes u_{\sigma(n)}\otimes\omega \\
&=s_{A,r}(\alpha)\sum_{\sigma\in S_n}\sgn(\sigma)u_{\sigma(1)}\otimes\cdots\otimes u_{\sigma(n)}\otimes\omega=\varphi(s_{A,r}(\alpha)),
\end{align*}
where the last equality again follows from \eqref{eq:phionendtensor}.

Having now shown that $\varphi(I_{A/R})=0$, we obtain an induced map $\overline{\varphi}\colon G(A/R)\to V^{\otimes n} \otimes \extp^n(V^*)$. Since $\pi$ is surjective, $\overline{\varphi}$ is as well. Since $\overline{\varphi}$ is a map of $(A^{\otimes n}*S_n)$-modules, to prove it is an isomorphism, it is enough to show it is an isomorphism of $R$-modules. To do so, we construct a surjective $R$-module map which is a section of $\overline{\varphi}$.

After choosing a basis of $V$, we may identify $A$ with $\Mat_n(R)$ and use the notation $e_{ij}$ to indicate a matrix that is $0$ in all entries except $1$ in the $(i,j)$-th position. It is clear that $G(A/R)$ is generated by (the image of) the elements $(e_{i_1j_1})^{(1)}\cdots(e_{i_nj_n})^{(n)}$ for $1 \leq i_k, j_k \leq n$.  We claim that for every function $\tau\colon\{1,2,\ldots,n\}\to \{1,2,\ldots,n\}$, we have the following equality in $G(A/R)$:
\begin{equation} \label{eq:eproducttau}
(e_{i_1\tau(1)})^{(1)}\cdots(e_{i_n\tau(n)})^{(n)}=
\begin{cases}
\sgn(\tau)\,(e_{i_1 1})^{(1)}\cdots(e_{i_n n})^{(n)} & \textrm{for } \tau\in S_n\\
0, & \textrm{for } \tau\notin S_n
\end{cases}
\end{equation}
Let $N_\tau=\sum_j e_{j\tau(j)}$. Then because $e_{i_k k} N_\tau = e_{i_k \tau(k)}$ and
$(N_\tau)^{(1)}\cdots(N_\tau)^{(n)}=\det(N_\tau)$ in $G(A/R)$, we have
$$
(e_{i_1 \tau(1)})^{(1)} \cdots (e_{i_n \tau(n)})^{(n)}
= e_{i_1 1}\cdot_1\cdots e_{i_n n}\cdot_n(N_\tau)^{(1)} \cdots (N_\tau)^{(n)}
= \det(N_\tau)(e_{i_1 1})^{(1)} \cdots (e_{i_n n})^{(n)}.
$$
Since $\det(N_\tau)$ vanishes for $\tau\notin S_n$ and is equal to $\sgn(\tau)$ for $\tau\in S_n$, we have \eqref{eq:eproducttau}. Thus, $G(A/R)$ is generated by the elements $(e_{i_1 1})^{(1)}\cdots(e_{i_n n})^{(n)}$ for $1 \leq i_k \leq n$.

Next, notice that $V^{\otimes n} \otimes \extp^n(V^*)$ is a free $R$-module of rank $n^n$ with basis $u_{i_1}\otimes\dots\otimes u_{i_n}\otimes\omega$, where $1\leq i_j\leq n$ and $\omega=u_1\wedge\cdots\wedge u_n$. We can therefore define an $R$-module map $\psi\colon V^{\otimes n} \otimes \extp^n(V^*)\to G(A/R)$ by
\[
\psi(u_{i_1}\otimes\dots\otimes u_{i_n}\otimes\omega)=(e_{i_1 1})^{(1)}\cdots(e_{i_n n})^{(n)}.
\]
By \eqref{eq:phionendtensor}, we see
\[
\overline\varphi((e_{i_1 1})^{(1)}\cdots(e_{i_n n})^{(n)})=\sum_{\sigma\in S_n}\sgn(\sigma)e_{i_11}(u_{\sigma(1)})\otimes\dots\otimes e_{i_nn}(u_{\sigma(n)})\otimes\omega=u_{i_1}\otimes\dots\otimes u_{i_n}\otimes\omega
\]
and so $\psi$ is a section of $\overline\varphi$. Since the $(e_{i_1 1})^{(1)}\cdots(e_{i_n n})^{(n)}$ generate $G(A/R)$ as an $R$-module, it follows that $\psi$ is surjective, hence an isomorphism of $R$-modules. Therefore, $\overline\varphi$ is an isomorphism as well.
\end{proof}

\subsection{Central simple algebras}
\label{subsec:csa}
We recall some basic facts about central simple algebras. For a field $F$, let $\Fsep$ denote its separable closure. If $\A$ is a central simple algebra of dimension $n^2$ over $F$, then there exists an $\Fsep$-vector space $V$ of dimension $n$ and an $\Fsep$-algebra isomorphism $\A_{\Fsep}:=\A\otimes_F \Fsep\simeq\End(V)$. Since the Galois group $G$ of $\Fsep/F$ acts continuously on $\A_{\Fsep}$, we obtain an induced action on $\End(V)$. In particular, Galois descent implies that giving a central simple algebra of dimension $n^2$ over $F$ is equivalent to giving a continuous $G$-action on $\End(V)$ over $\Fsep$.

In what follows, we endow $\A$ with the degree $n$ $F$-algebra structure from Example \ref{ex:central-simple-alg-polyn}, namely choose a finite Galois extension $K/F$, a $K$-algebra isomorphism $\iota'\colon\A\otimes_F K\stackrel{\simeq}{\longrightarrow}\Mat_n(K)$, and let $\iota$ be the embedding map $\A\to\Mat_n(K)$. Lemma \ref{l:base changing the degree n alg structure} tells us that $\A\otimes_F K$ inherits a degree $n$ $K$-algebra structure. As we will see momentarily, this is {\em not} the degree $n$ $K$-algebra structure on $\Mat_n(K)$ defined in Example \ref{ex:matrix-alg-polyn}; nonetheless, $\A\otimes_F K$ and $\Mat_n(K)$ do have isomorphic Galois closures, as we will see in the course of proving the following result.

\begin{lemma}
\label{l:CSA base extension noniso degn but same gc}
We have
\[
G(\A/F)\otimes_F \Fsep\simeq G(\A_{\Fsep}/\Fsep)\simeq V^{\otimes n}\otimes\extp^nV^*.
\]
\end{lemma}
\begin{proof}
Throughout the proof, we endow $\A\otimes_FK$ with its degree $n$ $K$-algebra structure coming from Lemma \ref{l:base changing the degree n alg structure}. By the base change theorem \ref{thm:base change}, it is enough to show that if we endow $\Mat_n(K)$ with the degree $n$ $K$-algebra structure coming from Example \ref{ex:matrix-alg-polyn}, then $\iota'\colon\A\otimes_FK\to\Mat_n(K)$ induces an isomorphism of Galois closures.

Recall that the degree $n$ $K$-algebra structure on $\A\otimes_FK$ is induced from the map $\iota\otimes1\colon\A\otimes_FK\to\Mat_n(K)\otimes_F K\simeq\Mat_n(K\otimes_FK)$. Let $\eta\colon\Mat_n(K\otimes_FK)\to\Mat_n(K)$ be the map induced from the $K$-algebra morphism $K\otimes_FK\to K$ sending $\alpha\otimes\beta$ to $\alpha\beta$. We claim that $\eta\circ(\iota\otimes1)=\iota'$. Indeed, let $a\in\A$, $\alpha\in K$, and $\iota'(a)=(Q_{ij})\in\Mat_n(K)$. Since $\iota'$ is a $K$-algebra map, we have $\iota'(a\otimes\alpha)=(Q_{ij}\alpha)$. On the other hand, $(\iota\otimes1)(a\otimes\alpha)=(Q_{ij})\otimes\alpha\in\Mat_n(K)\otimes_FK$ which is identified with the matrix $(Q_{ij}\otimes\alpha)\in\Mat_n(K\otimes_FK)$. This maps under $\eta$ to $(Q_{ij}\alpha)=\iota'(a\otimes\alpha)$, which proves our claim. It then follows immediately from Remark \ref{rmk:Galois closure depends weakly on Rprime} that $\iota'$ induces an isomorphism of Galois closures.
\end{proof}

In light of Lemma \ref{l:CSA base extension noniso degn but same gc}, by Galois descent, $G(\A/F)$ determines and is determined by a continuous $G$-action on $V^{\otimes n}\otimes\extp^n(V^*)$. The following result shows how to obtain this $G$-action in terms of the one on $\End(V)$, i.e., how to determine $G(\A/F)$ in terms of $\A$.

\begin{proposition}
\label{l:galois-action-on-galois-closure-of-EndV}
With notation as above, let $\varphi\colon G\to\aut(\End(V))$ be the continuous Galois action corresponding to the central simple algebra $\A$, and let $\pi\colon V^{\otimes n}\otimes (V^*)^{\otimes n}\to V^{\otimes n}\otimes \bigwedge^n (V^*)$ be the natural projection. Then for every $\sigma\in G$, the automorphism $\varphi(\sigma)^{\otimes n}$ of $\End(V)^{\otimes n} \simeq V^{\otimes n}\otimes (V^*)^{\otimes n}$ preserves $\ker\pi$, thereby inducing an automorphism $\psi(\sigma)$ of $V^{\otimes n}\otimes \bigwedge^n (V^*)$. The resulting map $\psi\colon G\to\aut(V^{\otimes n}\otimes \bigwedge^n (V^*))$ gives the Galois action corresponding to $G(\A/F)$.
\end{proposition}
\begin{proof}
By Theorem \ref{thm:base change}, we know that tensoring the surjection $\A\to G(\A/F)$ with $\Fsep$ yields the surjection $\A_{\Fsep}^{\otimes n}\to G(\A_{\Fsep}/\Fsep)$. As a result, the continuous $G$-action on $G(\A_{\Fsep}/\Fsep)$ is induced from that on $\A_{\Fsep}^{\otimes n}$. The $G$-action on $\A_{\Fsep}^{\otimes n}$ is nothing more than the one obtained from $\A_{\Fsep}$ by tensoring $n$ times, i.e., $\sigma\in G$ acts on $\A_{\Fsep}^{\otimes n}$ by $\varphi(\sigma)^{\otimes n}$. Finally, as noted above, the $G$-action on $G(\A_{\Fsep}/\Fsep)$ is induced from that on $\A_{\Fsep}^{\otimes n}$. Hence, each $\varphi(\sigma)^{\otimes n}$ preserves $I_{\A_{\Fsep}/\Fsep}=\ker\pi$ and the resulting action on $G(\A_{\Fsep}/\Fsep)$ is given by $\psi(\sigma)$.
\end{proof}

Let us next consider a large class of explicit examples which includes all central simple algebras over number fields (see, e.g., \cite[Chapter 15]{pierce} for further details). Let $K/F$ be a Galois extension having a cyclic Galois group of order $n$ with generator $\sigma$. For $\gamma \in F^*$, one may define an associative $F$-algebra
\begin{equation} \label{eq:cyclicalg}
\A(K/F,\sigma,\gamma) := (K\oplus uK\oplus\cdots u^{n-1}K) / (u^n=\gamma; \alpha u=u\sigma(\alpha)\,\forall\,\alpha \in K).
\end{equation}
The Albert--Brauer--Hasse--Noether theorem combined with the Grunwald--Wang theorem implies that if $F$ is a number field, then every central simple algebra over $F$ is a cyclic algebra, that is, of the form \eqref{eq:cyclicalg}. If $\gamma \in N_{K/F}(K^*)$, then $\A(K/F,\sigma,\gamma)$ is isomorphic to $\Mat_n(F)$. If $\gamma \in F^*$ modulo $N_{K/F}(K^*)$ has order exactly $n$ (that is, $\gamma^n\in N_{K/F}(K^*)$ and $\gamma^d\notin N_{K/F}(K^*)$ for all $d$ dividing $n$), then $\A(K/F,\sigma,\gamma)$ is a division algebra. 

Now fix a central simple algebra $\A$ over $F$ of the form \eqref{eq:cyclicalg}. Since $K$ splits $\A$, we have an injection of $\A$ into $\A \otimes_F K\simeq\Mat_n(K)$; explicitly $\A$ may be identified with the subring of $\Mat_n(K)$ consisting of elements of the form
$$
\begin{pmatrix}
x_0 & \gamma\sigma(x_{n-1}) & \gamma\sigma^2(x_{n-2}) &\cdots& \gamma\sigma^{n-1}(x_1)\\
x_1 & \sigma(x_{0}) & \gamma\sigma^2(x_{n-1}) &\cdots& \gamma\sigma^{n-1}(x_2)\\
x_2 & \sigma(x_{1}) & \sigma^2(x_{0}) &\cdots& \gamma\sigma^{n-1}(x_3)\\
\vdots & \vdots&\vdots& &\vdots \\
x_{n-1} & \sigma(x_{n-2}) &\sigma^2(x_{n-3}) &\cdots& \sigma^{n-1}(x_{0})
\end{pmatrix}.
$$

Since $\A_K\simeq\Mat_n(K)$, the central simple algebra $\A$ corresponds to a Galois action on $\Mat_n(K)$, which we now describe explicitly. Fix an $n$-dimensional $K$-vector space $V$ and basis $u_1, \ldots, u_n$ of $V$ to identify $\End(V) = \Mat_n(K)$. Recall the notation that $e_{ij}$ (or $e_{i,j}$) refers to the $n \times n$ matrix whose only nonzero entry is a $1$ in the $(i,j)$-th position. For all $\alpha\in K$, the action of $\Gal(K/F)=\langle\sigma\rangle$ is given by
\[
\sigma(\alpha e_{ij}):=
\begin{cases} 
      \sigma(\alpha) e_{i+1,j+1}, & i,j<n \\
      \sigma(\alpha) e_{11}, & i=j=n\\
      \sigma(\alpha) \gamma e_{1,j+1}, & i=n,j<n\\
      \sigma(\alpha) \gamma^{-1} e_{i+1,1}, & i<n,j=n
   \end{cases}
\]
In other words, $\sigma$ acts as usual on $K$, and it adds 1 to both of the $i$ and $j$ indices, multiplying by $\gamma$ or $\gamma^{-1}$ whenever the $i$ or $j$ index, respectively, overflows. Written in matrix form, we have
\[
\sigma\colon(\alpha_{ij})_{i,j}\longmapsto
\left(
\begin{array}{ccccc}
\sigma(\alpha_{nn}) & \gamma\sigma(\alpha_{n1}) & \gamma\sigma(\alpha_{n2}) &\cdots& \gamma\sigma(\alpha_{n,n-1})\\
\gamma^{-1}\sigma(\alpha_{1n}) & \sigma(\alpha_{11}) & \sigma(\alpha_{21}) &\cdots& \sigma(\alpha_{2,n-1})\\
\gamma^{-1}\sigma(\alpha_{2n}) & \sigma(\alpha_{21}) & \sigma(\alpha_{22}) &\cdots& \sigma(\alpha_{3,n-1})\\
\vdots & \vdots&\vdots& &\vdots \\
\gamma^{-1}\sigma(\alpha_{n-1,n}) & \sigma(\alpha_{n-1,1}) &\sigma(\alpha_{n-1,2}) &\cdots& \sigma(\alpha_{n-1,n-1})
\end{array}
\right)
\]

The induced action on $V^{\otimes n}\otimes \bigwedge^n V^*=G(\A_K/K)$ given by Proposition \ref{l:galois-action-on-galois-closure-of-EndV} is as follows: let $u_{n+1}:=u_1$ and $\omega=u_1^*\wedge\dots\wedge u_n^*$. Then for $1\leq i_j\leq n$, we have 
\[
\sigma\colon \alpha u_{i_1}\otimes\dots\otimes u_{i_n}\otimes\omega\,\longmapsto\, \sigma(\alpha)\gamma^{r-1}(-1)^{n-1} u_{i_1+1}\otimes\dots\otimes u_{i_n+1}\otimes\omega
\]
where $r$ is the number of $j$ such that $i_j=n$.

\begin{example}
Let $F/\QQ$ be a field that does not contain a square root of $-1$ and let $K=F(i)$. Then $\sigma(i)=-i$ and let $\overline\alpha := \sigma(\alpha)$. If $\gamma\in F^*$ is not of the form $a^2+b^2$ for $a,b\in F$, then $\A(K/F,\sigma,\gamma)$ is a division algebra, whose elements are explicitly represented by matrices of the form
\[
\left(
\begin{array}{cc}
x_0 & \gamma \overline x_1\\
x_1 & \overline x_0
\end{array}
\right)
\]
with $x_i\in K$. For example, when $F=\RR$ and $\gamma=-1$, this gives the usual matrix representation of the Hamiltonian quaternions $\mathbb{H}$.

By Proposition \ref{l:galois-action-on-galois-closure-of-EndV}, the action of $\Gal(K/F)=\<\sigma\>$ on $\Mat_2(K)$ corresponding to $G(\A/F)$ is given by
\[
\begin{split}
&\sigma((\alpha_{11}u_1\otimes u_1+\alpha_{12}u_1\otimes u_2+\alpha_{21}u_2\otimes u_1+\alpha_{22}u_2\otimes u_2)\otimes\omega) \\
=&-(\gamma^{-1}\overline\alpha_{11}u_2\otimes u_2+\overline\alpha_{12}u_2\otimes u_1+\overline\alpha_{21}u_1\otimes u_2+\gamma\overline\alpha_{22}u_1\otimes u_1)\otimes\omega.
\end{split}
\]
Since $G(\A/F)$ is the subspace of elements which are fixed by $\sigma$, it is explicitly given by the set of 
\[
(\alpha u_1\otimes u_1-\gamma\overline\alpha u_2\otimes u_2+\beta u_1\otimes u_2-\overline\beta u_2\otimes u_1)\otimes\omega
\]
with $\alpha,\beta\in K$. On the other hand, $\A$ is a quadratic algebra over $F$, so we know from Proposition \ref{prop:quadratic} that $G(\A/F)\simeq\A$, where the left $(\A^{\otimes 2}*S_2)$-module structure on $\A$ is described in \S\ref{subsec:quadratic algebras}. The map
\[
\left(
\begin{array}{cc}
\beta & \gamma\overline\alpha\\
\alpha & \overline\beta
\end{array}
\right)
\,\longmapsto\,
(\alpha u_1\otimes u_1-\gamma\overline\alpha u_2\otimes u_2+\beta u_1\otimes u_2-\overline\beta u_2\otimes u_1)\otimes\omega
\]
yields an explicit isomorphism between our two different descriptions of $G(\A/F)$ as a left $(\A^{\otimes 2}*S_2)$-module.
\end{example}

\subsection{Group rings}
An easy application of the product formula (Theorem \ref{thm:product-formula}) and our work in Sections \ref{subsec:endo-rings} and \ref{subsec:csa} 
allows us to compute Galois closures of group rings. For a finite group $G$ and a field $F$ with characteristic prime to $|G|$, Maschke's theorem implies that the group ring $F[G]$ is semisimple, that is, has the structure of a product of central simple algebras (and thus is a degree $n$ algebra for some integer $n$). One thus can compute the Galois closure of such a group ring. If all $G$-representations are split, the computation is easier:

\begin{proposition}
\label{prop:gp-algebras}
Let $G$ be a finite group and $F$ a field whose characteristic is prime to $|G|$. If all $G$-representations are split over $F$, then
\[
G(F[G]/F)\simeq(\bigotimes_\rho (V_\rho^{\otimes n_\rho} \otimes \extp^{n_\rho} V^*_\rho))^N
\]
where the tensor product varies over the irreducible representations $\rho\colon G\to V_\rho$, $n_\rho:=\dim V_\rho$, and $N$ is the multinomial coefficient $N:=|G|! / \prod_\rho n_\rho!$.
\end{proposition}
\begin{proof}
By Maschke's theorem and Artin-Wedderburn (see, e.g., \cite[Chapter 6]{serre-finitegps}), one has $F[G] \simeq \prod_\rho \End(V_\rho)$. The result then follows from Theorem \ref{thm:product-formula} and the computation of $G(\End(V_\rho)/F)$ in Theorem \ref{thm:gcforEndV}.
\end{proof}

\section{Hermitian representations from Galois closures} \label{sec:Hermitian}

As indicated in the introduction, one of our motivations for studying non-commutative Galois closures is related to constructing ``Hermitian'' representations.  We describe in \S \ref{sec:Hermcubes} one way to construct such representations and how Galois closures are needed.

\subsection{Definitions} \label{sec:Hermcubes}
We would like to study representations of algebraic groups that generalize tensor products of standard representations.

A simple explicit example is the space $W$ of $m \times m$ matrices that are {\em Hermitian} with respect to a quadratic algebra $A$ over $R$.  Using the notation for quadratic algebras from \S \ref{sec:exquadalgs}, these are $m \times m$ matrices $C = (c_{ij})$ such that $\overline{C} = C^t$, where $\overline{C}$ denotes the entrywise conjugate of $C$, i.e., $c_{ij} = \overline{c_{ji}}$ for all $1 \leq i,j \leq m$.  If $A$ is a rank $\rho$ free module over $R$, then $W$ is a free module over $R$ of rank $m + \rho m(m-1)/2$.  Moreover, the group $\GL_m(A)$ naturally acts on $W$: for $\gamma \in \GL_m(A)$, we define the action of $\gamma$ by $\gamma \cdot C = \gamma C \overline{\gamma}$ where $\overline{\gamma}$ is the entrywise conjugate of $\gamma$. For example, if $A$ is simply $R$ itself considered as a degree $2$ algebra with polynomial $P_r(T) = (T-r)^2$, then we have recovered the space of symmetric $m \times m$ matrices over $R$ with the standard action of $\GL_m(R)$; for $A = R^2$ as a quadratic algebra, this space is isomorphic to the space of $m \times m$ matrices over $R$ with the standard action of $\GL_m(R) \times \GL_m(R)$. It is of course possible to define $W$ in a basis-free manner, as we will see in more generality below.

We wish to generalize the above to a notion of a Hermitian $n$-dimensional $m \times \cdots \times m$ array with entries in a degree $n$ algebra $A$ over $R$.  Intuitively, we would like the symmetric group $S_n$ to act on such an array in two different ways: by an $S_n$-action on $A$ and by permuting the factors, and we would like to restrict to the arrays for which these two actions agree.  In general, an algebra $A$ does not come equipped with a natural $S_n$-action, but its Galois closure $G(A/R)$ does. So instead we allow for entries in $G(A/R)$. We now make this definition precise.

We begin with a coordinate-free description of the representation and then give an explicit description in terms of coordinates. Let $U$ be a free $R$-module of rank $m$. There are two natural $S_n$-actions on $G(A/R)\otimes_R U^{\otimes n}$, one is the left action on $G(A/R)$ and the other is a right action on $U^{\otimes n}$ given by permuting coordinates. For $\sigma \in S_n$ and $\aleph \in G(A/R) \otimes_R U^{\otimes n}$, we denote the two actions by $\sigma \cdot_1 \aleph$ and $\sigma \cdot_2 \aleph$, respectively.

\begin{definition}
\label{def:hermitian}
We define the \emph{Hermitian space} $\HH_{A,U}$ to be the subspace of $G(A/R)\otimes_R U^{\otimes n}$ where the two $S_n$-actions agree (up to an inverse), i.e.,
$$\{ \aleph \in G(A/R) \otimes_R U^{\otimes n} : \sigma \cdot_1 \aleph = \sigma^{-1} \cdot_2 \aleph \textrm{ for all } \sigma \in S_n \}.$$
\end{definition}

\begin{remark}
\label{rmk:Hermitianization as space of invariants}
We have two commuting (left) $S_n$-actions on $G(A/R)\otimes U^{\otimes n}$, hence an action of $S_n\times S_n$. The Hermitianization $\HH_{A,U}$ is the space of invariants for the diagonally embedded copy of $S_n$ in $S_n\times S_n$, i.e., $\HH_{A,U}=(G(A/R)\otimes U^{\otimes n})^{S_n}$.
\end{remark}

We have a natural action of $A \otimes_R \End(U)$ on $G(A/R) \otimes_R U^{\otimes n}$, where $A$ acts on $G(A/R)$ via the diagonal embedding of $A$ in $A^{\otimes n}$ and $\End(U)$ acts on each factor of $U^{\otimes n}$ in the standard way.

\begin{lemma}
\label{l:Herm-rep}
The action of $A\otimes_R \End(U)$ on $G(A/R)\otimes_R U^{\otimes n}$ commutes with the two $S_n$-actions on $G(A/R) \otimes_R U^{\otimes n}$. In particular, $\HH_{A,U}$ is preserved by the action of $A\otimes_R \End(U)$.
\end{lemma}

\begin{proof}
This follows from the $A \otimes_R \End(U)$-action being the same on every factor $A$ of $G(A/R)$ and every factor $U$ of $U^{\otimes n}$ by definition. 
\end{proof}

\begin{definition}
We refer to $\HH_{A,U}$ with this action of $A\otimes\End(U)$ as a \emph{Hermitian representation}.
\end{definition}

We next describe the Hermitian representation explicitly using coordinates. After choosing a basis for $U$, we may make identifications $U\simeq R^{\oplus m}$ and $A\otimes_R \End(U)\simeq \Mat_m(A)$. Letting $T=\{(i_1,\ldots,i_n)\mid 1\leq i_j\leq m\}$, we then have
\[
G(A/R)\otimes_R U^{\otimes n}\simeq G(A/R)^{\oplus T},
\]
i.e., elements of $G(A/R) \otimes U^{\otimes n}$ are represented as $n$-dimensional $m \times \cdots \times m$ arrays with entries in $G(A/R)$, where $T$ parametrizes the coordinates of the array. The two $S_n$-actions can then by described as follows. The first $S_n$-action on $G(A/R)$ acts on $M\in G(A/R)^{\oplus T}$ coordinate-wise: $\sigma(M_t)_{t\in T}=(\sigma(M_t))_{t\in T}$. The second $S_n$-action on $U^{\otimes n}$ yields an action on $T$ via $\sigma(i_1,\ldots,i_n):=(i_{\sigma(1)},\ldots,i_{\sigma(n)})$, which then induces an action on $G(A/R)^{\oplus T}$. The Hermitian space $\HH_{A,U}$ is then the subspace of arrays where these two actions agree (up to an inverse).

In coordinates, the action of $A \otimes \End(U)$ on $\HH_{A,U}$ is described as follows. Let $\gamma\in A\otimes_R \End(U)\simeq\Mat_m(A)$ and $M = (M_t)_{t \in T} \in \HH_{A,U}$. Then $\gamma\cdot M$ has $(i_1, \ldots, i_n)$-entry as follows:
\[
(\gamma\cdot M)_{i_1\ldots i_n}=\sum_{j_1, \ldots, j_n}(\gamma_{i_1j_1}\otimes\gamma_{i_2j_2}\otimes\cdots\otimes\gamma_{i_nj_n})\cdot M_{j_1\ldots j_n};
\]
here $(\gamma_{i_1j_1}\otimes\gamma_{i_2j_2}\otimes\cdots\otimes\gamma_{i_nj_n})\cdot M_{j_1\ldots j_n}$ comes from the left action of $A^{\otimes n}$ on $G(A/R)$.

\subsection{Product formula for Hermitian representations}
In this subsection we show that if $A$ is a product of degree $n_i$ algebras, then the Hermitian representation associated to $A$ is a tensor product of the corresponding Hermitian representations.

\begin{theorem}[Product formula for Hermitianizations]
\label{thm:prod-form-Herm}
For $1\leq i\leq\ell$, let $A_i$ be a degree $n_i$ $R$-algebra. If $A=A_1\times\cdots\times A_\ell$, then $\HH_{A,U}\simeq\HH_{A_1,U}\otimes_R \cdots\otimes_R \HH_{A_\ell,U}$.
\end{theorem}
\begin{proof}
For notational simplicity, we reduce to the case $\ell = 2$. Recall from Theorem \ref{thm:product-formula} that if $A_1$ and $A_2$ are $R$-algebras of degree $n_1$ and $n_2$, respectively, then the $R$-algebra $A = A_1 \times A_2$ has degree $n = n_1 + n_2$ and its Galois closure is given by
\[
G(A/R)=\Ind_{S_{n_1}\times S_{n_2}}^{S_n}(G(A_1/R) \otimes G(A_2/R)).
\]
Further recall from Remark \ref{rmk:Hermitianization as space of invariants} that $\HH_{A,U}=(G(A/R)\otimes U^{\otimes n})^{S_n}$ for the diagonally embedded copy of $S_n$.

Now, in complete generality, if $G$ is a finite group, $H\subseteq G$ is a subgroup, $V$ is a finite-dimensional $H$-representation, and $W$ is a finite-dimensional $G$-representation, then by the projection formula (see, e.g., \cite[\href{https://stacks.math.columbia.edu/tag/01E6}{Tag 01E6}]{stacks-project}) we have
\[
\Ind_H^G(V)\otimes W\simeq\Ind_H^G(V\otimes\Res^G_H(W)).
\]
Taking $G$-invariants of both sides yields
\[
(\Ind_H^G(V)\otimes W)^G\simeq(V\otimes\Res^G_H(W))^H.
\]
Applying this to the case where $G=S_n$, $H=S_{n_1}\times S_{n_2}$, $V=G(A_1/R)\otimes G(A_2/R)$, and $W=U^{\otimes n}$, we see
\[
\HH_{A,U}=(G(A/R)\otimes U^{\otimes n})^{S_n}\simeq (G(A_1/R)\otimes G(A_2/R)\otimes U^{\otimes n_1}\otimes U^{\otimes n_2})^{S_{n_1}\times S_{n_2}}\simeq \HH_{A_1,U}\otimes\HH_{A_2,U}
\]
thereby proving the result.
\end{proof}

\begin{example}
Viewing $R$ as a degree $1$ algebra over itself, we have $\HH_{R,U}=U$. The Hermitian representation is given by the standard action of $\GL(U)$ on $U$. Applying the product formula (Theorem \ref{thm:prod-form-Herm}) tells us that for the degree $n$ algebra $A = R^n$, the Hermitian representation is $\HH_{A,U}=U^{\otimes n}$ equipped with the natural action of $\GL(U)^n$.
\end{example}

\begin{example}
\label{ex:mxmxm as m tuple of mxm}
If $B$ is a quadratic $R$-algebra and $A$ is the cubic algebra $R \times B$ as in Section \ref{sec:cubicfromsmaller}, we conclude that $\HH_{A,U} \simeq U \otimes \HH_{B,U}$. With a choice of basis for the rank $m$ free $R$-module $U$, we see that $m \times m \times m$ arrays that are Hermitian with respect to $A = R \times B$ may be viewed as an $m$-tuple of $m \times m$ matrices that are Hermitian with respect to $B$. More generally, for a degree $n$ $R$-algebra $B$, elements of the Hermitian space $\HH_{R \times B,U}$ may be viewed as an $m$-tuple of elements of $\HH_{B,U}$.
\end{example}

\subsection{Example: endomorphism algebras (or matrix rings)}

Let $V$ be a rank $n$ free $R$-module and let $A=\End(V)$.  We study the Hermitian representation $\HH_{A,U}$, where $U$ is a rank $m$ free $R$-module.

\begin{proposition}
\label{prop:WforEndV}
Let $U$ be a rank $m$ free $R$-module, $V$ be a rank $n$ free $R$-module, and $A=\End(V)$.
The Hermitian representation $\HH_{A,U}$ is isomorphic to $\extp^n(V\otimes U)\otimes\extp^n(V^*)$.
\end{proposition}
\begin{proof}
As mentioned in Remark \ref{rmk:Hermitianization as space of invariants}, we may view $\HH_{A,U}$ as the invariant subspace of a particular $S_n$-action.
Recall from Theorem \ref{thm:gcforEndV} that $G(A/R) \simeq V^{\otimes n}\otimes\extp^n(V^*)$. Since $\extp^n(V^*)$ is the sign representation, our desired subspace of $G(A/R) \otimes U^{\otimes n}$ is therefore $\extp^n(V^*)$ tensored with the copy of the sign representation in $V^{\otimes n}\otimes U^{\otimes n}=(V\otimes U)^{\otimes n}$. This is given by $\extp^n(V\otimes U)\otimes\extp^n(V^*)$.

The action of $A \otimes \End(U) = \End(V) \otimes \End(U)$ on $\HH_{A,U} \simeq \extp^n(V\otimes U)\otimes\extp^n(V^*)$ is then given by $\End(U)$ acting on $U$ and $\End(V)$ acting on both $V$ and $V^*$.
\end{proof}

\begin{example}
\label{ex:examples coming from Hermitianization of matrix algebras}
If $U$ and $V$ are free $R$-modules of ranks $m$ and $n$, respectively, then with a choice of basis for each, we observe that $\HH_{\End(V),U}$ is naturally isomorphic to (a twist of) the $n$-th wedge product of a free $R$-module of rank $mn$, with the standard action of $\GL_{mn}(R)$. In particular, we obtain some interesting representations (see \S \ref{sec:vinberg}):
\begin{equation}
\renewcommand{\arraystretch}{1.5}
\begin{array}{cccc}
\dim U &  \dim V & \HH_{\End(V),U} & \textrm{group} \\
\hline
m & 2 & \extp^2(2m) & \GL_{2m} \\
2 & 3 & \extp^3(6) & \GL_6 \\
2 & 4 & \extp^4(8) & \GL_8 \\
3 & 3 & \extp^3(9) & \GL_9 \\
\end{array} 
\end{equation}
Note that the first two cases in the table could have been computed without the definition of a Galois closure. The first may be visualized as Hermitian $m \times m$ matrices over the quadratic algebra $\Mat_2(R)$. The second may be visualized as $2 \times 2 \times 2$ cubes, Hermitian over the cubic algebra $\Mat_3(R)$; the orbits of this space for $R = \ZZ$ are studied in \cite{hcl1}, which motivated much of this paper.
\end{example}

\subsection{Vinberg representations} \label{sec:vinberg}

In \cite{vinberg}, Vinberg considers finite $d$-gradings of Lie algebras $\mathfrak{g} = \sum_{i=0}^{d-1} \mathfrak{g}_i$ and studies the representation of $G_0 \subset G$ on $\mathfrak{g}_1$, where $G_0\subset G$ are the groups corresponding to the Lie algebras $\mathfrak{g}_0 \subset \mathfrak{g}$. The orbit spaces of many of these representations have, in recent years, been studied as moduli spaces of arithmetic or algebraic data \cite{hcl1, jthorne-vinbergAIT, pollack, coregular, rainssam1, rainssam2}.

We observe that many of these representations may be viewed as Hermitian representations, as in Example \ref{ex:examples coming from Hermitianization of matrix algebras}. Below is a table with Vinberg's representations, coming from a $d$-grading on an exceptional group $G$, that also arise as Hermitian representations $\mathcal{H}_{A,U}$ for an $m$-dimensional vector space $U$ over a field $k$ and a degree $n$ $k$-algebra $A$. The last column refers to references where the representation and/or corresponding moduli problem is studied.

\begin{equation*} \def\arraystretch{1.5}
\begin{array}{c|ccccccccc}
 & G & d & \textrm{(semisimple) group} & \textrm{representation} & m & n & A & \textrm{reference}\\
\hline
1. & E_6^{(1)} & 2 & \SL_2 \times \SL_6 & 2 \otimes \extp^3(6) & 2 & 4 & k \times \Mat_{3}(k) & \textrm{\cite[\S 6.6.2]{coregular}}\\
2. & E_6^{(1)} & 3 & \SL_3 \times \SL_3 \times \SL_3 & 3 \otimes 3 \otimes 3 & 3 & 3 & k^3 & \textrm{\cite[\S 4.2]{coregular}}\\
3. & E_7^{(1)} & 2 & \SL_8 & \extp^4(8) & 2 & 4 & \Mat_4(k)& \textrm{\cite{jthorne-vinbergAIT}}\\
4. & E_7^{(1)} & 3 & \SL_3 \times \SL_6 & 3 \otimes \extp^2(6) & 3 & 3 & k \times \Mat_2(k) & \textrm{\cite[\S 5.5]{coregular}}\\
5. & E_8^{(1)} & 2 & \SL_2 \times E_7 & 2 \otimes 56 & 2 & 4 & k \times \mathcal{J} &\textrm{\cite[\S 6.6.3]{coregular}} \\
6. & E_8^{(1)} & 3 & \SL_9 & \extp^3(9) & 3 & 3 & \Mat_3(k) & \textrm{\cite{rainssam1,rainssam2,romanothorne-E8}}\\
7. & E_8^{(1)} & 3 & \SL_3 \times E_6 & 3 \otimes 27 & 3 & 3 & k \times \mathbb{O} &\textrm{\cite[\S 5.4]{coregular}}\\
8. & F_4^{(1)} & 3 & \SL_3 \times \SL_3 & 3 \otimes \Sym^2(3) & 3 & 3 & k \times k_{[2]} & \textrm{\cite[\S 5.2.1]{coregular}}\\
9.& G_2^{(1)} & 2 & \SL_2 \times \SL_2 & 2 \otimes \Sym^3(2) & 2 & 4 & k \times k_{[3]} & \textrm{\cite[\S 6.3.2]{coregular}}\\
10. &D_4^{(3)} & 3 & \SL_3 & \Sym^3(3) & 3 & 3 & k_{[3]} & \textrm{\cite[\S 5.2.2]{coregular}}\\
\end{array}
\end{equation*}

We use the notation $k_{[n]}$ to denote $k$ as a degree $n$ algebra (see Example \ref{ex:triv-generic-polyn}). The notation $\mathbb{O}$ refers to the split octonion algebra over $k$, and $\mathcal{J}$ is the exceptional cubic Jordan algebra (also known as the space of Hermitian $3 \times 3$ matrices with respect to $\mathbb{O}$). In both of these cases, this algebra $A$ is non-associative, but the Galois closure is not needed to describe the Hermitian space (see \cite{coregular} for details).

There are in fact many additional cases of representations that arise from Vinberg's construction and are closely related to Hermitian representations. A simple example is that of $\SL_2 \times \SL_4$ acting on $2 \otimes \Sym^2(4)$, which comes from a $4$-grading of $E_6^{(2)}$; it is the tensor product of a $2$-dimensional space and the space of Hermitian matrices over $k_{[2]}$. However, there are still numerous Vinberg representations for which we do not yet have such an interpretation.

Due to the non-associativity of general cubic Jordan algebras, our current definition of Galois closure does not apply. However, given the above connection between Vinberg representations and Galois closures, we ask:

\begin{question}
Can our definition of Galois closure be extended to the case of cubic Jordan algebras (or other non-associative algebras)? If so, can any of the remaining Vinberg representations be recovered as Hermitian representations associated to such algebras?
\end{question}

\begin{remark}
As mentioned in the introduction, our original motivation for this paper was to study as many representations as possible, including those arising from Vinberg theory, with a uniform method. In particular, we hope it will be possible to study the moduli problems coming from these Hermitian spaces using uniform geometric constructions via Galois closures.
\end{remark}

\newpage
\bibliographystyle{amsalpha}
\bibliography{gcbib}

\end{document}